\documentclass[11pt]{amsart}
\usepackage{amssymb,amsmath,latexsym,times,enumitem,tikz, mathtools,ytableau}
\usepackage{chessfss, comment}
\usepackage{thmtools, thm-restate}
\usepackage{hyperref}

\newtheorem{prop}{Proposition}[section] 
\newtheorem{thm}[prop]{Theorem}
\newtheorem{conj}[prop]{Conjecture}
\newtheorem{lemma}[prop]{Lemma} 
\newtheorem{cor}[prop]{Corollary}
\theoremstyle{definition} 
\newtheorem{definition}[prop]{Definition}
\newtheorem{remark}[prop]{Remark}
\newtheorem{example}[prop]{Example}

\newcommand{\op}{\operatorname}
\newcommand{\Det}{\op{Det}}
\newcommand{\CCC}{\mathcal{C}}
\newcommand{\FF}{\mathbb{F}}
\newcommand{\inv}{^{-1}}
\newcommand{\im}{\op{Im}}
\renewcommand{\ker}{\op{Ker}}
\newcommand{\brac}[1]{\langle #1 \rangle}
\newcommand{\Span}[1]{\op{Span}\{ #1 \}}
\newcommand{\ZZ}{\mathbb{Z}}
\newcommand{\Symm}{\mathfrak{S}}
\newcommand{\Rook}{\raisebox{-7pt}{\BlackRookOnWhite}}
\newcommand{\Out}{\raisebox{-7pt}{\larger\BlackEmptySquare}}

\newcommand*{\vertbar}{\rule[-.5ex]{0.5pt}{2ex}}
\newcommand{\Inv}{\op{Inv}}
\newcommand{\GL}{\op{GL}}
\newcommand{\ol}{\overline}
\newcommand{\transpose}{^{\mathrm{T}}}
\newcommand{\wt}{\op{wt}}
\newcommand{\defn}{\textbf}

\begin{document}

\author[B.~Koprowski]{Brandon Koprowski}
\address{Department of Mathematics, North Carolina State University, Raleigh, NC, USA}
\email{bmkoprow@ncsu.edu}
\author[J.~B.~Lewis]{Joel Brewster Lewis}
\address{Department of Mathematics, George Washington University, Washington, DC, USA}
\email{jblewis@gwu.edu}

\title{Enumeration of Nondegenerate $2 \times (k+1) \times k$ Hypermatrices}
\date{February 25, 2026}

\begin{abstract}
We consider the problem of enumerating hypermatrices of format $2 \times (k + 1) \times k$ over a finite field that have nonzero hyperdeterminant and whose nonzero entries are restricted to a plane partition.  We conjecture an attractive product formula for the enumeration, and prove it in many cases.  In general, we show that the enumeration is given (up to a power of $q - 1$) by a polynomial in $q$ with nonnegative integer coefficients, whose value at $q = 1$ enumerates a natural family of three-dimensional rook placements.
\end{abstract}
\maketitle

\section{Introduction}

In this paper we attempt to generalize the classical $q$-rook theory of matrices to higher dimensions. We begin with a brief survey of the classical theory which will serve as a starting point for our generalization. 

Counting permutations as rook placements was first systematically studied by Kaplansky and Riordan \cite{kaplansky}, who showed a large family of permutation enumeration questions can be phrased in the terms of rook theory. Here a \defn{rook} can be thought of as a chess rook (which can attack vertically and horizontally) and a \defn{rook placement} is a placement of non-attacking rooks on a subset of a chessboard; for a more precise statement see Definition~\ref{def:rookPlacement}. 
Rook theory was further developed in a series of five papers by Goldman, Joichi, and White. In particular, their first paper \cite{goldman} developed a complete understanding of when two partition-shaped boards have the same number of rook placements for all numbers of rooks. The ``$q$" in $q$-rook theory was first studied by Garsia and Remmel \cite{garsia}, who introduced the \defn{$q$-rook number}
\[
R_k(B; q) = \sum_{C} q^{\op{inv}(C,B)}
\]
of a board $B$.  In this formula, the sum is over the set of rook placements $C$ of $k$ rooks on the partition-shaped board $B$, and $\op{inv}(C, B)$ is a certain \emph{inversion statistic} associated to each placement. 

In addition to rook numbers, one can also look at \defn{hit numbers} for a particular board contained in the $n$-by-$n$ square: the $k$th hit number is the number of $n$-rook placements on the square in which exactly $k$ rooks are placed in $B$. Garsia and Remmel were using the language of rook theory as a means to solve a separate problem but in the process discovered that the \emph{$q$-hit number} for partition-shaped boards is a polynomial with positive coefficients. However, they were unhappy with the proof of this fact and asked for others to provide a more satisfying explanation.
 
In response to this call for a better proof, Haglund \cite{haglund} gave an interpretation of $q$-rook and $q$-hit numbers in terms of matrices over finite fields with entries forced to be zero.\footnote{An alternative proof, via other methods, was given around the same time by Dworkin \cite{dworkin}.}  In particular, Haglund devised a way to relate each placement of $k$ rooks on a board $B$ to a collection of matrices over a finite field $\FF_q$ of rank $k$. He then used this to prove that the number of $n \times n$ matrices of rank $k$ with support corresponding to $B$, which he called $P_k(B)$, is related to the $q$-rook number $R_k(B; q)$ by 
\begin{equation}\label{eq:Haglund} P_k(B) = (q-1)^kq^{|B| - k}R_k(B; q\inv).\end{equation} 
In the particular case that $k = n$, the expression on both sides of this equation factors in an attractive way (see Theorem~\ref{thm:haglund}).  

The main goal of this paper is to generalize this connection between the number of rook placements and the number rank-$k$ matrices over a finite field to higher dimensions by enumerating \emph{hypermatrices}.  Our work is inspired by an unpublished note by Musiker and Yu \cite{musiker}, who showed that the number of hypermatrices (see Definition~\ref{def:hypermatrix}) over $\FF_q$ with nonzero \emph{hyperdeterminant}, the multidimensional analogue of the determinant, factors nicely in the $2\times 2 \times 2$ case. It is worth noting that there are at least three inequivalent ways to define a hyperdeterminant; we will be using Cayley's second definition, commonly denoted $\Det$, whose important properties are outlined below in Corollary~\ref{cor:zero}. If a hypermatrix has nonzero hyperdeterminant we say it is \defn{nondegenerate}.
We denote by $\GL_k(\FF_q)$ the general linear group of $k \times k$ invertible matrices over $\FF_q$. 

It is easy to show that
$|\GL_2(\FF_q)| = (q^2 -1)(q^2 - q)$.
If we let $[n]_q = 1 + q + q^2 + \cdots + q^{n-1}$ and $[n]!_q = [n]_q[n-1]_q \cdots [1]_q$ then we can write this more succinctly as \[|\GL_2(\FF_q)| = q(q-1)^2[2]!_q.\] 
Musiker and Yu proved the number of nondegenerate $2 \times 2 \times 2$  hypermatrices over $\FF_q$ is $$(q^4 - 1)(q^4 - q^3),$$ which led them to conjecture the number of nondegenerate hypermatrices of dimension $k + 1$ in the general $2 \times 2 \times \cdots \times 2$ case is $(q^{2^k} - 1)(q^{2^k} -q^{2^k -1})$. However, they were unsuccessful in even proving the $2\times 2 \times 2 \times 2$ case due to the fact that the hyperdeterminant is this case is unapproachable, having over $2$ million terms. 
 
It turns out that, rather than the ``cubic'' hypermatrices considered by Musiker and Yu, the easier hypermatrices to study are those of \defn{boundary format}, which are hypermatrices of shape 
 $$(k_1 + 1) \times (k_2 + 1) \times \cdots \times (k_r + 1),$$
 with $k_i = k_1 + k_2 + \cdots + k_{i-1} + k_{i+1} + \cdots + k_{r}$ for some $ 1 \leq i \leq r$.
It was shown by Aitken \cite{aitken} that for dimension $3$, the question of whether or not a boundary format hypermatrix is nondegenerate reduces to a question of whether or not a certain family of matrices are full rank (see Lemma~\ref{lem:faceSum}). This is helpful for enumeration but it is actually better than that: using this fact Aitken showed that there exists a free and transitive group action on the set $M_k(\FF_q)$ of nondegenerate $2\times (k+1) \times k$ hypermatrices with entries in $\FF_q$ and using this he was able to show that
$$|M_k(\FF_q)| = q^{k^2}(q-1)^{2k}[k+1]!_q[k]!_q.$$ We pick up where Aitken left off, refining this count for boundary format hypermatrices with restricted entries.  

We begin in Section~\ref{sec:preliminaries} by reviewing the basic definitions and concepts of hypermatrices and the hyperdeterminant. We then show there exists an analogous concept to matrices with an integer partition-shaped section forced to be zero for boundary format hypermatrices. Specifically, these are hypermatrices with a plane partition section forced to be zero. We also show that there is a unique maximal (by inclusion) plane partition such that nondegenerate hypermatrices exist when the entries of the plane partition are forced to be zero. 

In Section~\ref{sec:main conjecture}, we state our main conjecture (Conjecture~\ref{mainConjecture}): for $2 \times (k+1) \times k$ hypermatrices and a plane partition $P = \brac{\lambda, \mu}$ where $\lambda$ and $\mu$ are integer partitions of length $k$ (see Definition~\ref{def:plane partition} for this notational convention), the number of nondegenerate hypermatrices with support avoiding $P$ is 
\[ q^{k^2} (q-1)^{2k}[k+1 - \lambda_1]_q \cdots [2- \lambda_k]_q \cdot [k-\mu_1]_q \cdots [1 -\mu_k]_q.\] 
Although we are unable to prove our conjecture for all $P$, we show that it holds for several natural families of plane partitions, including plane partitions of one layer, the largest and smallest plane partitions for which the question is interesting, and all plane partitions in the case $k = 2$.

In Section~\ref{sec:4}, we provide further indirect evidence in favor of Conjecture~\ref{mainConjecture}.  We introduce a three-dimensional analogue of rook placements that we call \emph{hyperrook placements}.  We show that the number of hyperrook placements on a $2 \times (k+1) \times k$ board avoiding a plane partition $P = \brac{\lambda, \mu}$ is exactly the value
\[ (k+1 - \lambda_1) \cdots (2 - \lambda_k) \cdot (k - \mu_1) \cdots (1 - \mu_k)\] 
that one would predict by substituting $q = 1$ into Conjecture~\ref{mainConjecture}.  Further, in Section~\ref{sec:Bruhat}, we make a concrete connection between these hyperrook placements and nondegenerate hypermatrices concrete, giving an analogue of Haglund's theorem \eqref{eq:Haglund}.  Our main result (Theorem~\ref{thm:weak form of conjecture}) is to establish the following weak version of Conjecture~\ref{mainConjecture}: for every plane partition $P$, the number of nondegenerate hypermatrices over $\FF_q$ avoiding $P$ is given by a polynomial in $q$ that (up to factors of $q - 1$) has positive-integer coefficients and whose value at $q = 1$ is the number of hyperrook placements avoiding $P$.

We conclude the paper with Section~\ref{sec:final remarks} with some final remarks and open questions.

\section{Hypermatrices and Nondegeneracy}
\label{sec:preliminaries}

\begin{definition}\label{def:hypermatrix}A \defn{hypermatrix} of format $(k_1 + 1) \times (k_2+1) \times \cdots \times (k_r + 1)$ over a field $\FF$ is an $r$-dimensional array of elements $a_{i_1, \ldots, i_r} \in \FF$ with $1 \leq i_j \leq k_j+1$ for $j = 1, \ldots, r$.\end{definition}

 A \defn{slice} of a hypermatrix is an $(r-1)$-dimensional subarray which is obtained by holding one index constant. Furthermore, a \defn{slice in the $i$th direction} is the subarray where we hold the $i$th index constant. For instance, the $2 \times (k+1) \times k$ hypermatrix $(a_{i_1,i_2,i_3})$ with $1 \leq i_1 \leq 2$, \ $1 \leq i_2 \leq k+1$, and $1 \leq i_3 \leq k$ has two $(k+1) \times k$ slices in the first direction: $(a_{1,i_2,i_3})$ and $(a_{2,i_2,i_3})$. For a hypermatrix of dimension $(k_1 + 1) \times (k_2+1) \times \cdots \times (k_r + 1)$ we can see there are $k_i + 1$ slices in the $i$th direction, each of dimension
$$(k_1 + 1) \times \cdots \times (k_{i-1} + 1) \times (k_{i+1}+1) \times \cdots \times (k_r+1).$$ We call any two distinct slices in the same direction \defn{parallel}.

In this paper we will be especially concerned with hypermatrices of boundary format $2 \times (k+1) \times k$ which we will view as the pair $(M_1, M_2)$ of $(k+1) \times k$ slices in the $1$st direction. We will sometimes refer to $M_1$ as the \defn{front face} and $M_2$ as the \defn{back face}. For instance, the following is a representation of a $2 \times 4 \times 3$ hypermatrix:
\[ \begin{bmatrix}1&0&0\\0&1&0\\0&0&1\\0&0&0\end{bmatrix}\begin{bmatrix}0&0&0\\1&0&0\\0&1&0\\0&0&1\end{bmatrix}.\]

The collection of $(k_1 + 1) \times (k_2+1) \times \cdots \times (k_r + 1)$ hypermatrices over a field $\FF$ comes equipped with the natural action of the group \[G = \GL_{k_1 + 1}(\FF) \times \GL_{k_2 + 1}(\FF) \times \cdots \times \GL_{k_r + 1}(\FF),\] defined as follows. Let $H$ be a hypermatrix of format $(k_1 + 1) \times \cdots \times (k_r+1)$ with $A_i \in \GL_{k_i + 1}(\FF)$ and $M_1, M_2, \ldots, M_{k_i + 1}$ the slices of $H$ in the $i$th direction. Let $A = (a_{j_1,j_2})$ with $1\leq j_1,j_2 \leq k_i + 1$, then 
\[(1, \ldots, A_i, \ldots 1)\circ H\] is the hypermatrix with slices $M_1', M_2', \ldots, M_{k_i+1}'$ in the $i$th direction where 
\[M'_{n} = a_{n,1}M_1 + a_{n,2}M_2 + \cdots + a_{n,k_i + 1}M_{k_i + 1}.\]
In other words, $(1, \ldots, 1, A_i, 1, \ldots, 1) \in G$ acts on $H$ by taking linear combinations of the $k_i + 1$ slices in the $i$th direction in the coefficients of $A_i$. In the case of $2 \times (k+1) \times k$ hypermatrices represented as a pair of $(k+1) \times k$ matrices, the $\GL_{k+1}(\FF_q)$-action looks like simultaneous row operations on $M_1$ and $M_2$, and the $\GL_k(\FF_q)$-action looks like simultaneous column operations. For instance, 
\[\left(\begin{bmatrix} 1&0\\ 0&1 \end{bmatrix},\begin{bmatrix} a&0&0&0\\0&b&0&0\\ 0&0&c&0\\ 0&0&0&d\end{bmatrix}, \begin{bmatrix}0&0&x\\1&0&0\\ 0&1&0\end{bmatrix}\right) \circ \begin{bmatrix}1&0&0\\0&1&0\\0&0&1\\0&0&0\end{bmatrix}\begin{bmatrix}0&0&0\\1&0&0\\0&1&0\\0&0&1\end{bmatrix} = \begin{bmatrix}0&a&0\\ 0&0&b\\ cx&0&0\\ 0&0&0 \end{bmatrix} \begin{bmatrix} 0&0&0\\ 0&b&0\\ 0&0&c\\ dx&0&0\end{bmatrix}.\] 

There is a function on hypermatrices called the \defn{hyperdeterminant} that plays a similar role to the determinant for matrices. In particular we call hypermatrices with zero hyperdeterminant \defn{degenerate} and correspondingly hypermatrices with nonzero hyperdeterminant \defn{}{nondegenerate}.  To distiguish between the hyperdeterminant and matrix determinant we will denote the hyperdeterminant of a hypermatrix $M$ by $\Det(M)$ and determinant of a matrix $A$ by $\det(A)$. The formal definition of the hyperdeterminant is technical and outside the scope of this paper so instead we will describe its properties; for a complete definition see \cite[p.~444]{gkz}. Although the hyperdeterminant is defined for hypermatrices of any format, there are numerous formats for which every hypermatrix is degenerate, as the next theorem shows.

\begin{thm}[{\cite[Thm.~1.3]{gkz}}]
The hypdeterminant of format $(k_1 + 1) \times \cdots \times (k_r +1)$ is non-trivial if and only if 
\[k_j \leq \sum_{i \ne j} k_i\] for all $j = 1, \ldots, r$.
\end{thm}

We now restrict our attention to hypermatrices of non-trivial format. Let us fix a non-trivial format $(k_1 + 1) \times (k_2 + 1) \times \cdots \times (k_r + 1)$.

\begin{prop}[{\cite[Prop.~1.4]{gkz}}]\label{prop:invariant}
The hyperdeterminant is relatively invariant under the group action \[G = \GL_{k_1 + 1}(\FF) \times \cdots \times \GL_{k_r + 1}(\FF),\] that is, there exist integers $n_1, n_2, \ldots, n_r$ such that if $M$ is a hypermatrix and $(A_1, \ldots, A_r) \in G$ then \[\Det \big((A_1, \ldots, A_r)\circ M\big) = \det(A_1)^{n_1}\cdots \det(A_r)^{n_r}\cdot \Det(M).\]
\end{prop}

The previous proposition leads to the following corollaries; part (d) in particular will be helpful when we classify hypermatrices with certain entries forced to be zero. 

\begin{cor}[{\cite[Cor. 1.5]{gkz}}]\label{cor:zero} \hfill
\begin{enumerate}[label = (\alph*)]
\item{Interchanging two parallel slices leaves the hyperdeterminant invariant up to sign.}
\item{The hyperdeterminant is a homogeneous polynomial in the entries of each slice. The degree of homogeneity is the same for parallel slices.}
\item{The hyperdeterminant does not change if we add to some slice a scalar multiple of a parallel slice.}
\item{The hyperdeterminant of a hypermatrix with two parallel slices proportional to each other is zero. In particular, the hyperdeterminant of a hypermatrix with a zero slice is zero.}
\end{enumerate}
\end{cor}

For the rest of the paper we will only consider boundary format hypermatrices of three dimensions, that is, hypermatrices of format $(k_1 + 1) \times (k_1 + k_2 + 1) \times (k_2 + 1)$ for positive integers $k_1$ and $k_2$. We restrict our attention to these hypermatrices in order to make use of the next lemma, which, in conjunction with Corollary~\ref{cor:zero}(d), will allow us to classify which hypermatrices with restricted entries are nondegenerate.
\begin{lemma}[{\cite[Lemma 2.3]{aitken}}]\label{lem:faceSum}
Let $M$ be a hypermatrix of boundary format $(k_1 + 1) \times (k_1+k_2 + 1) \times (k_2 + 1) $ over a field $\FF$ and let $M_1, M_2, \ldots, M_{k_1 + 1}$ be the slices in the first direction. Then $M$ is nondegenerate if and only if every linear combination 
$$c_1M_1 + c_2M_2 + \cdots +c_{k_1+1}M_{k_1+1}$$ with $c_1, c_2, \ldots, c_{k_1+1} \in \ol{\FF}$ not all zero is full rank.
\end{lemma}

We now take a brief detour from talking about hypermatrices to formalize some concepts from rook theory that we will be trying to generalize. For positive integers $m$ and $n$ we the define the $m \times n$ rectangle $\mathcal{R}_{m \times n}$ to be the subset \[\{(i,j) : 1 \leq i \leq m \textnormal{ and } 1\leq j \leq n\} \subseteq \ZZ_{>0} \times \ZZ_{> 0}\] where we label rows of $\ZZ_{>0} \times \ZZ_{> 0}$ top to bottom and columns left to right.
For instance, $\mathcal{R}_{5 \times 4}$ with this labeling scheme looks like
\[ \ytableausetup
{mathmode, boxframe=normal, boxsize=2em}\begin{ytableau}
\scriptstyle(1,1) & \scriptstyle(1,2) & \scriptstyle(1,3)&\scriptstyle(1,4) \\
\scriptstyle(2,1) &\scriptstyle(2,2) & \scriptstyle(2,3)&\scriptstyle(2,4) \\
\scriptstyle(3,1) &\scriptstyle(3,2) & \scriptstyle(3,3)&\scriptstyle(3,4) \\
\scriptstyle(4,1) &\scriptstyle(4,2) & \scriptstyle(4,3)&\scriptstyle(4,4) \\
\scriptstyle(5,1) &\scriptstyle(5,2) & \scriptstyle(5,3)&\scriptstyle(5,4) 
\end{ytableau}.\]
The main idea from rook theory we are trying to generalize is about rook placements on \emph{integer partition-shaped boards}. We formalize these in a complementary manner by looking at rectangles with an integer partition-shaped section removed.

\begin{definition}\label{def:integerPartition}
Let $m$ and  $n$ be positive integers. An \defn{integer partition} $\lambda = \brac{\lambda_1, \lambda_2, \ldots, \lambda_n}$ of format $m \times n$ is a list of nonnegative integers $\lambda_i \leq m$ such that $\lambda_{i+1} \leq \lambda_{i}$ for $i = 1, 2, \ldots, n-1$. \end{definition}

For example, the integer partition $\lambda = \brac{4,2,1,0}$ is of format $4 \times 4$. We represent each partition by its Young diagram in transposed French notation; for example, the diagram of $\lambda$ is
$$\ydiagram{1,1,2,3}.$$ 
We will write $\lambda = \emptyset$ to mean that $\lambda$ is the empty partition 
$\lambda_1 = \lambda_2 = \cdots = \lambda_n = 0.$

For an integer partition $\lambda$ of format $m \times n$ we define the \defn{board} $B_{\lambda}$ as 
\[ B_{\lambda} = \{(i,j) : 1 \leq i \leq m - \lambda_j \textnormal{ and } 1\leq j \leq n \}\subseteq \mathcal{R}_{m\times n}.\] If $\lambda$ is clear from context we will often just write $B$.
Intuitively, $B_{\lambda}$ is the result of starting with the rectangle $\mathcal{R}_{m\times n}$, lining up the bottom left hand corner of $\lambda$ with $\mathcal{R}_{m \times n}$, and then eliminating all the squares belonging to their union.
 For example, let $\lambda= \brac{4,2,1,0}$ as above. Then we view the board $B_{\lambda}$ as the white squares in the picture below:
$$
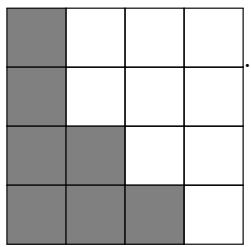
\begin{ytableau}
*(gray) & & & \\
*(gray) & & & \\
*(gray)&*(gray) & & \\
*(gray)&*(gray) & *(gray)& \\
\end{ytableau}.$$

\begin{definition}\label{def:rookPlacement}A \defn{rook placement} on $B$ is a subset $S$ of $B$ such that for any $(i_1,j_1), (i_2,j_2) \in S$ we have $i_1 \ne i_2$ and $j_1 \ne j_2$. 
\end{definition}
Intuitively we view a rook placement as a placement of non-attacking rooks in the squares of $B$ or equivalently a placement of rooks with at most one rook in each row and column. For instance, letting $B_{\lambda}$ be as above the rook placement $\{ (2,2), (1,3), (4,4)\}$ corresponds to
$$
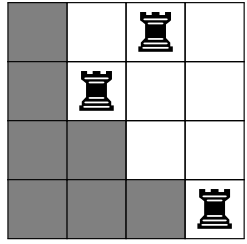
\begin{ytableau}
*(gray) & & \Rook& \\
*(gray) &\Rook& & \\
*(gray)&*(gray) & & \\
*(gray)&*(gray) & *(gray)&\Rook \\
\end{ytableau}$$
where \Rook represents a rook.

Now let us consider the set $R^{n\times n}$ of rook placements of cardinality $n$ on $\mathcal{R}_{n \times n}.$ Let $S \in R^{n\times n}$ be some rook placement and let $(i_j, j) \in S$ be the pair with second coordinate $j$. Define $f:R^{n\times n} \to \Symm_n$ by letting $f(S)$ be the permutation with $f(S)(j) = i_j$ for all $j \in [n]$.
If we then let $f\inv : \Symm_n \to R^{n \times n}$ be the function defined by $f\inv(\sigma) = \{(\sigma(j),j): j \in [n]\}$ then we can see that $f\inv$ is the well defined inverse of $f$, so $f$ is a bijection. Thus there is a natural bijective correspondence between rook placements on an $n \times n$ board and permutations of $\Symm_n$. For example, in the $4\times 4$ case the rook placement 
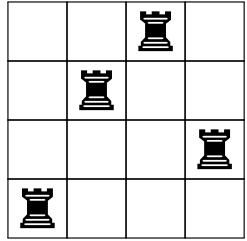
\[\begin{ytableau}
\empty& & \Rook& \\
& \Rook& & \\
& & &\Rook \\
\Rook& & & \\
\end{ytableau} \]
 corresponds to the permutation $4213$ (written as a word in one-line notation). 

\begin{definition}\label{def:PermutationMatrix}
Let $\sigma \in \Symm_n$ be a permutation. We define the $n \times n$ \defn{permutation matrix} $m_{\sigma}$ to be the matrix with $(m_{\sigma})_{i,j} = 1$ if $i = \sigma(j)$ and zero otherwise.
\end{definition}

With the preceding definition in mind it is clear that the bijection $f$ between rook placements and permutations extends to a bijection between rook placements and permutation matrices. 
For instance, the matrix
\[\begin{bmatrix}
0&0&1&0\\
0&1&0&0\\
0&0&0&1\\
1&0&0&0
\end{bmatrix}\] corresponds to the permutation $4213$ and previous rook placement.
One will notice for a given rook placement $S$ on $B_{\lambda}$ the following inequality holds: $f(S)(j) \leq n - \lambda_j$ and therefore $(m_{f(S)})_{i,j} = 0$ for all $i > n - \lambda_j$. The main result from classical $q$-rook theory we are trying to generalize comes from looking at all matrices with entries $(i,j)= 0$ for all $i > n - \lambda_j$ over a finite field $\FF_q$.

\begin{definition}\label{def:respects}
Let $\lambda$ be an integer partition of format $m \times n$ and $A$ an $m\times n$ matrix with entries in $\FF$. We say that $A$ \defn{respects} $\lambda$ if $a_{i,j} = 0$ for all $i > m-\lambda_j$ for $1 \leq j \leq n$.
\end{definition}

It immediately follows from the previous definition that if $\lambda$ and $\mu$ are integer partitions such that $\mu_i \leq \lambda_i$ for all $1 \leq i \leq n$ and a matrix $A$ respects $\lambda$ then $A$ respects $\mu$. Naturally, we will call $\mu$ a \defn{sub-partition} of $\lambda$ and write $\mu \subseteq \lambda$. 

The following result is an explicit product formula for the full-rank case of the expressions appearing in \eqref{eq:Haglund} (Haglund's theorem \cite[Thm.~1]{haglund}), which relates the number of rook placements on $B_{\lambda}$ to the number of invertible matrices over $\FF_q$ that respect $\lambda$.  Recall $[k]_q = 1 + q + \cdots + q^{k-1}$ for $k \geq 1$. 

\begin{thm}\label{thm:haglund}
    Let $\lambda$ be an integer partition with $\lambda_i \leq n-i$ for $i \in [n]$.
    The number of placements of $n$ rooks on $B_{\lambda}$ is $(n - \lambda_1)(n-1-\lambda_2) \cdots (1-\lambda_n)$ and the number of $n \times n$ invertible matrices over $\FF_q$ that respect $\lambda$ is 
    \[q^{\binom{n}{2}}(q-1)^n[n-\lambda_1]_q[n-1-\lambda_2]_q \cdots [1 - \lambda_n]_q.\]
\end{thm}

These formulas are easy to prove recursively, by building the rook placement or matrix column by column from the left. 
In order to extend this connection to hypermatrices, we need an analogous notion of an integer partition for three dimensions. 
 
\begin{definition}\label{def:plane partition}
A \defn{plane partition} $P$ of format $(k_1 + 1) \times (k_1+k_2 + 1) \times (k_2 + 1)$ is a collection $\brac{\lambda^{(1)}, \lambda^{(2)}, \ldots, \lambda^{(k_1 + 1)}}$ of integer partitions 
$\lambda^{(i)}$ of format $(k_1+k_2 + 1) \times (k_2 + 1)$ such that $\lambda^{(i+1)} \subseteq \lambda^{(i)}$ for all $1 \leq i \leq k_1$.
\end{definition} 

Analogous to the matrix case, we say a hypermatrix $M$ of format $(k_1 + 1) \times (k_1+k_2 + 1) \times (k_2+ 1)$ \defn{respects} a plane partition $P = \brac{\lambda^{(1)}, \lambda^{(2)}, \ldots, \lambda^{(k_1 + 1)}}$ if each slice $M_i$ in the first direction respects $\lambda^{(i)}$ for $1 \leq i \leq k_1 + 1$. For instance, if $\brac{\lambda^{(1)}, \lambda^{(2)}}$ is a plane partition of format $2 \times 5 \times 4$ with $\lambda^{(1)} = \brac{4,2,1,0}$ and $\lambda^{(2)} = \brac{3,2,1,0}$ then any hypermatrix of the form 
$$\begin{bmatrix}*&*&*&*\\0&*&*&*\\0&*&*&*\\0&0&*&*\\0&0&0&*\end{bmatrix}\begin{bmatrix}*&*&*&*\\ *&*&*&*\\0&*&*&*\\0&0&*&*\\0&0&0&*\end{bmatrix}$$
respects $\brac{\lambda^{(1)}, \lambda^{(2)}}$. For convenience, if there exists a nondegenerate hypermatrix $M$ that respects a plane partition $P$ we say $P$ \defn{allows} nondegenerate hypermatrices.
Similar to how we have sub-partitions for integer partitions there is an analogous notion for plane partitions.
\begin{definition}\label{def:sub-partition}
Let $P$ and $P'$ be plane partitions of format $(k_1 + 1) \times (k_1 + k_2 + 1) \times (k_2 +1)$ where $P = \brac{\lambda^{(1)}, \lambda^{(2)}, \ldots, \lambda^{(k_1 + 1)}}$ and $P' = \brac{\mu^{(1)}, \mu^{(2)}, \ldots, \mu^{(k_1 + 1)}}$. We write $P' \subseteq P$ and say $P'$ is a \defn{sub-partition} of $P$ if $\mu^{(i)} \subseteq \lambda^{(i)}$ for all $1 \leq i \leq k_1 + 1$.   
\end{definition}

The next lemma follows immediately from the definitions.
\begin{lemma}\label{lem:subPartition}
If $P$ and $P'$ are plane partitions with $P' \subseteq P$ and $P$ allows nondegenerate hypermatrices then $P'$ allows nondegenerate hypermatrices.
\end{lemma}
Equivalently if $P'$ does not allow nondegenerate hypermatrices then $P$ does not allow nondegenerate hypermatrices.

\begin{lemma}
Let $P$ be a plane partition of format $(k_1 +1) \times (k_1 + k_2 + 1) \times (k_2 + 1)$ with $\lambda_{k_2+1}^{(k_1 + 1)} > 0$. Then $P$ does not allow nondegenerate hypermatrices.
\end{lemma}
\begin{proof}
By Lemma~\ref{lem:subPartition} we may assume $\lambda_{k_2 +1}^{(k_1 + 1)} = 1$. By Definition~\ref{def:integerPartition}, $\lambda_{k_2 + 1}^{(k_1 + 1)} \leq \lambda_{i}^{(k_1 + 1)}$ for $1 \leq i \leq k_2 + 1$. Furthermore, $\lambda^{(k_1 + 1)} \subseteq \lambda^{(j)}$ for all $1 \leq j \leq k_1 + 1$ and thus 
\[ 1 = \lambda_{k_2 + 1}^{(k_1 + 1)} \leq \lambda_{i}^{(j)}\] for all $1 \leq i \leq k_2 + 1$ and $1 \leq j \leq k_1 + 1$.
Let $M = (a_{i_1,i_2,i_3})$ be a hypermatrix that respects $P$, so by definition $a_{i_1, i_2, i_3} = 0$ whenever $i_2 > k_1 + k_2 + 1 - \lambda_{i_3}^{(i_1)}$. Therefore $a_{i_1,(k_1 + k_2+1), i_3} = 0$ for all $1\leq i_1 \leq k_1 + 1$ and $1 \leq i_3 \leq k_2 + 1$. However, this is the same as saying the $(k_1 + k_2 + 1)$ slice in the $2$nd direction is zero and by Corollary~\ref{cor:zero}(d) we conclude that $M$ is degenerate.
\end{proof}
By Lemma~\ref{lem:faceSum} we get the following corollary.
\begin{cor}\label{cor:DegenerateSum}
Let $P$ be a plane partition of format $(k_1 + 1) \times (k_1 + k_2 + 1) \times (k_2 + 1)$ with $\lambda_{k_2+1}^{(k_1 + 1)} > 0$ and $M$ a hypermatrix that respects $P$ with slices $M_1, M_2, \ldots, M_{k_1 +1}$ in the first direction. There exists some $c_1, c_2, \ldots, c_{k_1 + 1} \in \ol{\FF}_q$, not all zero, such that 
$$c_1M_1 + c_2M_2 + \cdots + c_{k_1 + 1}M_{k_1 + 1}$$ has less than full rank.
\end{cor}

It appears now that we have defined an analogous version of the partition-shaped restrictions on matrices for hypermatrices. For this to be truly analogous in the $q$-rook counting sense there is one property we would like to ensure holds. 
Let us go back to classical $q$-rook theory for a moment and fix a format $n \times n$. By basic linear algebra the partition $\delta = \brac{n-1, n-2, \ldots, 1,0}$ is maximal such that there exists an invertible matrix which respects $\delta$. In the $4\times 4$ case all matrices of form 
\[\begin{bmatrix} *&*&*&*\\ 0&*&*&*\\0&0&*&*\\0&0&0&* \end{bmatrix}\] respect $\delta$, and in particular the standard $4 \times 4$ identity matrix respects $\delta$.  We now define the corresponding objects in three dimensions, and show that the analogous result holds for them.

\begin{definition}\label{E}
Fix a format $(k_1 +1) \times (k_1 + k_2 +1) \times (k_2 + 1)$. 
We denote by $\Delta_{k_1,k_2}$ the plane partition $\langle \lambda^{(1)}, \ldots, \lambda^{(k_1 + 1)}\rangle$ where $\lambda^{(i)}_j = (k_1 + k_2 + 1) - (i + j -1)$ for $1 \leq i \leq k_1 + 1$ and $1 \leq j \leq k_2 + 1$.  Further, we denote by $E$ the hypermatrix with entry $E_{i_1,i_2,i_3}$ equal to $1$ when $i_2 = i_1 + i_3 -1$ and $0$ otherwise.
\end{definition}
\begin{lemma}[{\cite[Lemma 3.4]{gkz}}]\label{lem:Enondegenerate}
For any format $(k_1 +1) \times (k_1 + k_2 +1) \times (k_2 + 1)$, the hypermatrix $E$ is nondegenerate. 
\end{lemma} 


\begin{thm}\label{thm:MaxPlanePartition}
Let $\Delta_{k_1,k_2}$ be as in Definition~\ref{E}.  Then $\Delta_{k_1,k_2}$ is the unique maximal plane partition with respect to the ordering in Definition~\ref{def:sub-partition} allowing nondegenerate hypermatrices.
\end{thm} 
\begin{proof} By definition $\Delta_{k_1,k_2}$ allows $E$ which is nondegenerate by Lemma~\ref{lem:Enondegenerate}. Now we show if $Q \not\subseteq \Delta_{k_1, k_2}$ then $Q$ does not allow nondegenerate hypermatrices.

If $Q = \brac{ \mu^{(1)}, \ldots, \mu^{(k_1 + 1)}}$ is some plane partition of format $(k_1 + 1) \times (k_1 + k_2 + 1)\times (k_2 + 1)$ such that $Q \not\subseteq \Delta_{k_1,k_2}$, then there exists some $i,j > 0$ such that $\mu^{(i)}_j > (k_1 + k_2 + 1) - (i+j -1).$ Let $P = \brac{\lambda^{(1)}, \ldots, \lambda^{(i)}, \emptyset, \ldots, \emptyset}$ where $\lambda^{(1)} = \cdots = \lambda^{(i)}$ and 
$$\lambda^{(i)}_1 = \cdots = \lambda^{(i)}_j = (k_1 + k_2 + 1) - (i+j -1) + 1$$ and $\lambda^{(i)}_{j+1} = \cdots = \lambda^{(i)}_{k_2 + 1} = 0$, so $P\subseteq Q$. By Lemma~\ref{lem:subPartition} to prove $Q$ does not allow nondegenerate hypermatrices it suffices to show $P$ does not allow nondegenerate hypermatrices. Let $M = (a_{i_1,i_2,i_3})$ be a hypermatrix which respects $P$. Let $N$ be the sub-hypermatrix of format 
$$[(i-1) + 1] \times [(i-1) + (j-1) + 1] \times [(j-1) + 1]$$ that consists of the first $i+j-1$ rows and $j$ columns of the first $i$ slices of $M$ in the first direction. So, $N = (a_{i_1,i_2,i_3})$ where $1 \leq i_1 \leq i$, $1\leq i _2 \leq i + j -1$, and $1 \leq i_3 \leq j$. 
By the assumption that $M$ respects $P$, we know $(a_{i_1,i_2,i_3}) = 0$ for $1 \leq i_1 \leq i$, $1 \leq i_3 \leq j$, and 
$$ i_2 > (k_1 + k_2 + 1) - [(k_1 + k_2 + 1) - (i+j-1) + 1] = i+j -2.$$ Therefore $a_{i_1,(i+j-1),i_3} = 0$ for all $1 \leq i_1 \leq i$ and $1\leq i_3 \leq j$, or equivalently the $i + j -1$ slice of $N$ in the second direction is zero. By Corollary~\ref{cor:DegenerateSum} there exists 
$c_1, c_2, \ldots, c_i \in \ol{\FF}_q$ not all zero such that 
$$c_1N_1 + c_2N_2 + \cdots + c_iN_i$$ is less than full rank. Let us denote the result of this linear combination as the $(i+j-1) \times j$ matrix of column vectors $[v_1\mid v_2 \mid \cdots \mid v_j]$. Now we take linear combination
$$c_1M_1 + \cdots + c_iM_i + 0M_{i+1} + \cdots + 0M_{k_1 + 1}$$ of slices of $M$ in the first direction and let us denote this sum as the $(k_1 + k_2 + 1) \times (k_2 + 1)$ matrix of column vectors $[u_1\mid u_2 \mid \cdots \mid u_{k_2 + 1}]$. The first $j$ columns of this matrix are  
\[\begin{bmatrix}v_1\\ 0 \\ \vdots\\ 0\end{bmatrix},\begin{bmatrix}v_2\\ 0 \\ \vdots\\ 0\end{bmatrix}, \ldots, \begin{bmatrix}v_j\\ 0 \\ \vdots\\ 0\end{bmatrix}\]
which are linearly dependent. Thus the first $j$ columns of $[u_1\mid u_2 \mid \cdots \mid u_{k_2 + 1}]$ have rank less than $j$ and thus the matrix has rank less than $k_2 + 1$. It follows by Lemma~\ref{lem:faceSum} that $M$ is degenerate. Thus $P$ does not allow nondegenerate hypermatrices and the theorem follows.
\end{proof}

\section{Main Conjecture}
\label{sec:main conjecture}
In this section we are only concerned with hypermatrices of format $2 \times (k+1) \times k$. When it is clear from context we  will refer to the plane partition $\Delta_{1,(k-1)}$ as $\Delta$ and an arbitrary plane partition $\brac{\lambda^{(1)},\lambda^{(2)}}$ as $\brac{\lambda, \mu}$ where $\lambda$ and $\mu$ are integer partitions of format $(k+1) \times k$. Our main goal in this section is to introduce and study the following conjectural enumeration of nondegenerate hypermatrices of format $2 \times (k+1) \times k$ that avoid a particular plane partition $P \subseteq \Delta$.   

\begin{conj}\label{mainConjecture}
Fix a format $2\times (k+1) \times k$ and let $P = \brac{\lambda, \mu}$ be a plane partition such that $P \subseteq \Delta$. The number of nondegenerate  $2\times (k+1) \times k$ hypermatrices over $\FF_q$ which respect $P$ is 
\[q^{k^2}(q-1)^{2k}[k+1 - \lambda_1]_q[k - \lambda_2]_q \cdots [2 - \lambda_{k}]_q \cdot [k - \mu_1]_q[(k-1) - \mu_2]_q \cdots[1 - \mu_k]_q.\] 
\end{conj}

Although we will not be able to prove this conjecture for all $P$, there are large families of $P$ for which we will show it holds. 
But, before we get into hypermatrix counting we need to establish a few crucial definitions and a theorem from Aitken that will be the main tool in counting such objects.

\begin{definition}\label{G}
Fix a field $\FF$ and a format $2 \times (k+1) \times k$. Let $M_k$ be the set of nondegenerate hypermatrices over $\FF$ of format $2 \times (k+1) \times k$. Define the group 
\[G = \GL_{k+1}(\FF) \times \GL_{k}(\FF) / N\] where $N$ is the subgroup of $\GL_{k+1}(\FF) \times \GL_k(\FF)$ of ordered pairs $(c I_{k+1}, c\inv I_{k})$ for $c \in \FF^{\times}$.
\end{definition} 

We need to be careful about what we say next because this quotient group $G$ does not automatically inherit a group action on $M_k$ from $\GL_2(\FF) \times \GL_{k+1}(\FF) \times \GL_k(
\FF)$; however in this case we claim that it does. First, notice $\GL_{k+1}(\FF) \times \GL_k(\FF)$ inherits the group action of $\GL_2(\FF) \times \GL_{k+1}(\FF) \times \GL_k(\FF)$ by identifying $\GL_{k+1}(\FF) \times \GL_{k}(\FF)$ with the subgroup $1 \times \GL_{k+1}(\FF) \times \GL_k(\FF)$ where $1$ is the trivial subgroup of $\GL_2(\FF)$. Second, every element of $N$ commutes with every element of $\GL_{k+1}(\FF) \times \GL_k(\FF)$, so $N$ is in the center of $\GL_{k+1}(\FF) \times \GL_k(\FF)$. Therefore $N$ is normal and the group $G$ is well defined. Notice also that every element of $N$ acts trivially on the set $M_k$. Now choose an element $(g_1,g_2) N \in G$, for $(g_1,g_2) \in \GL_{k+1}(\FF) \times \GL_k(\FF)$, and a coset representative $(g_1,g_2) \cdot (cI_{k+1}, c\inv I_k)$. We define the action of $(g_1,g_2)N$ on $M_k$ by the action the representative $(g_1,g_2)\cdot (cI_{k+1}, c\inv I_k)$ inherits from $\GL_{k+1}(\FF) \times \GL_k(\FF)$. This is well defined since for any other coset representative and $M \in M_k$
\[ (g_1,g_2)\cdot (bI_{k+1}, b\inv I_k) \circ M = (g_1,g_2) \circ M = (g_1,g_2)\cdot (cI_{k+1}, c\inv I_k) \circ M \] since the elements of $N$ act trivially on $M_k$. Therefore the action of $\GL_{k+1}(\FF) \times \GL_k(\FF)$ on $M_k$ extends to a well defined action of $G$.

\begin{thm}[{\cite[Thm. 2.2]{aitken}}]\label{thm:aitken}
The action of $G$ on $M_k$ is a free, transitive action.
\end{thm}

It follows immediately that there is a bijection between $G$ and $M_k$ and we get the following corollary, which is also the case $\lambda = \mu = \emptyset$ of Conjecture~\ref{mainConjecture}.

\begin{cor}[{\cite[Prop.~3.1]{aitken}}]\label{cor:aitkenCorollary}
The number of nondegenerate $2 \times (k+1) \times k$ hypermatrices over $\FF_q$ is \[q^{k^2}(q-1)^{2k}[k+1]!_q[k]!_q.\]
\end{cor}
The previous corollary follows by computing the size of $G$ but also leads to a general idea of how to count nondegenerate hypermatrices using the action of $G$. Specifically, if we want to count the nondegenerate hypermatrices that respect a plane partition $P\subseteq \Delta$ we first look at the set \[S = \{(g_1,g_2) \in \GL_{k+1}(\FF_q) \times \GL_k(\FF_q) : (g_1,g_2) \circ E \textnormal{ respects } P\},\] where $E$ is the hypermatrix in Definition~\ref{E}. Then the number of nondegenerate hypermatrices that respect $P$ is $|S|/|N| = |S|/(q-1)$. We will apply this idea numerous times when counting families of nondegenerate hypermatrices.

Now we will prove the following two technical lemmas which will allow us to extend the special case above to the much broader family of plane partitions with $\mu = \emptyset$ and $\lambda$ arbitrary.

\begin{lemma}\label{lem:orthogonal}

Let $\{v_1,v_2, \ldots, v_n\}$ be a linearly independent collection of vectors in $\FF_q^{k+1}$ for $1 \leq n \leq k+1$. For a vector $w \in \FF_q^{k+1}$ let $w \cdot v$ be the standard dot product. There are $q^{k+1-n}$ vectors $x \in \FF_q^{k+1}$ such that $x\cdot v_i = 0$ for $1 \leq i \leq n$.
\end{lemma}
\begin{proof}
For a vector $v \in \FF_q^{k+1}$ define $v\transpose$ to be the transpose of $v$. We define the linear transformation $T_A: \FF_q^{k+1} \to \FF_q^{k+1}$ by the matrix of row vectors
\[ A = \begin{bmatrix} \text{---} \hspace{-0.2cm} & v_1\transpose & \hspace{-0.2cm} \text{---} \\ 
 & \vdots & \\
\text{---} \hspace{-0.2cm} & v_n\transpose & \hspace{-0.2cm} \text{---} \\ \text{---} \hspace{-0.2cm} & 0 & \hspace{-0.2cm} \text{---} \\
&\vdots& \\
\text{---} \hspace{-0.2cm} & 0 & \hspace{-0.2cm} \text{---}
\end{bmatrix}\]  where $0$ represents the zero vector.
For a vector $x \in \FF_q^{k+1}$ one can see
\[T_A(x) = Ax =  \begin{bmatrix} x\cdot v_1 & x\cdot v_2 & \cdots & x\cdot v_n & 0 & \cdots & 0 \end{bmatrix}\transpose\] 
and notice that $\ker(T_A)$ is the set of vectors we are trying to count. By the rank-nullity theorem $\dim \ker(T_A) = k+1 - \dim \im(T_A)$ and we claim that $\dim\im(T_A)= n$. To see this recall $\dim\im(T_A)$ is equal to the dimension of the row space of $A$. Clearly the dimension of the row space is $n$, so $\dim \im(T_A) = n$ and $\dim\ker(T_A) = k+1 - n$. Therefore $|\ker(T_A)| = q^{k+1- n}$ as claimed.
\end{proof} 
\begin{lemma}\label{lem:partitionTransposition}
Let $\lambda$ be an integer partition of format $(k+1) \times k$ with $\lambda_i \leq k+1 - i$ and $\lambda\transpose = \brac{\lambda_1\transpose, \lambda_2\transpose, \ldots, \lambda_{k+1}\transpose}$ be the transpose of $\lambda$. Then
\[ [k+1 - \lambda_1]_q[k - \lambda_2]_q \cdots [2 - \lambda_{k}]_q = [k+1 - \lambda_1\transpose]_q[k - \lambda_2\transpose]_q \cdots [1 - \lambda_{k+1}\transpose]_q.\]
\end{lemma}
\begin{proof}
To prove this we count nondegenerate $(k+1) \times (k+1)$ matrices that respect $\lambda$ in two different ways. Fix a prime power $q$.  By Theorem~\ref{thm:haglund} there are 
\[q^{\binom{k+1}{2}}(q-1)^{k+1}[k+1 - \lambda_1]_q[k - \lambda_2]_q \cdots [2 - \lambda_{k}]_q\] such matrices over $\FF_q$, and we have dropped the factor $[1 - \lambda_{k+1}]_q$ since by assumption $\lambda_{k+1} = 0$.
Next, we count the same objects except first we conjugate by the anti-diagonal matrix of all ones, and then take the transpose. It is easy to see this map is a bijection between $(k+1) \times (k+1)$ invertible matrices that respect $\lambda$ and those that respect $\lambda^T$. Thus, by Theorem~\ref{thm:haglund} there are
\[ q^{\binom{k+1}{2}}(q-1)^{k+1}[k+1 - \lambda_{1}\transpose ]_q[k - \lambda_{2}\transpose]_q\cdots[1 - \lambda_{k+1}\transpose]_q\]
such matrices. 
We have counted the same object two different ways so the displayed equations are equal.  Since they are equal for each prime power $q$, they are equal as polynomials, and the statement follows after canceling powers of $q$ and $q-1$.
\end{proof}

\subsection{The case \texorpdfstring{$\mu = \emptyset$}{mu is empty}}

Our next result shows that Conjecture~\ref{mainConjecture} holds for all $P \subseteq \Delta$ with $\mu = \emptyset.$
\begin{thm}\label{thm:muEmpty}
Let $P \subseteq \Delta$ with $\mu = \emptyset$. The number of nondegenerate hypermatrices which respect $P$ is 
\[q^{k^2}(q-1)^{2k}[k+1 - \lambda_1]_q[k - \lambda_2]_q \cdots [2 - \lambda_k]_q \cdot [k]!_q.\]
\end{thm}
\begin{proof}

By Theorem~\ref{thm:aitken} every nondegenerate hypermatrix can be written as $(g_1,g_2) \circ E$ for $(g_1, g_2) \in \GL_{k+1}(\FF_q) \times \GL_k(\FF_q)$. We count the number of pairs $(g_1,g_2)$ such that $(g_1,g_2) \circ E$ respects $P$. In particular, we begin by fixing a $g_2 \in \GL_{k}(\FF_q)$ and then compute the number of $g_1$ such that $(g_1,g_2)\circ E$ respects $P$.

Fix a format $2 \times (k+1)\times k$ and let $P \subseteq \Delta$ be a plane partition with $\mu = \emptyset$. 
For some $g_2 \in \GL_k(\FF_q)$ let $M$ be the hypermatrix $M = (1,g_2) \circ E$. We denote the entries of $M$ as $M_{i_1,i_2,i_3}$ where $i_1 \in [2]$, $i_2 \in [k+1]$, and $i_3 \in [k]$.
Let $g_1 \in \GL_{k+1}(\FF_q)$ such that $(g_1,1) \circ M$
respects $P$ and notice $(g_1,1) \circ M = (g_1,g_2) \circ E$. We will denote $g_1$ by the matrix of row vectors \[g_1 = \begin{bmatrix} \text{---} \hspace{-0.2cm} & v_1 & \hspace{-0.2cm} \text{---} \\ 
 & \vdots & \\
\text{---} \hspace{-0.2cm} & v_{k+1} & \hspace{-0.2cm} \text{---} \\
\end{bmatrix}\] and by assumption $g_1 \in \GL_{k+1}(\FF_q)$ the collection $\{v_1, v_2, \ldots, v_{k+1}\}$ is a linearly independent set of vectors.
Let us represent the hypermatrix $M$ as the pair of slices in the first direction $(M_1, M_2)$ and we will refer to each as a matrix of column vectors 
\[ M_1 = \begin{bmatrix}     \vertbar & \vertbar &        & \vertbar \\
    a_{1,1}    & a_{1,2}    & \ldots & a_{1,k}    \\
    \vertbar & \vertbar &        & \vertbar \end{bmatrix}, \; M_2 = \begin{bmatrix}     \vertbar & \vertbar &        & \vertbar \\
    a_{2,1}    & a_{2,2}    & \ldots & a_{2,k}    \\
    \vertbar & \vertbar &        & \vertbar \end{bmatrix}.\]
By Theorem~\ref{thm:aitken} the hypermatrix $M$ is nondegenerate, so by Lemma~\ref{lem:faceSum} the collections of vectors $\{a_{1,1}, a_{1,2}, \ldots, a_{1,k}\}$ and $\{a_{2,1}, a_{2,2}, \ldots, a_{2,k}\}$ are both linearly independent.
By the definition of the action, entry $(i_1,i_2,i_3)$ of $(g_1,1)\circ M$ is equivalent to $v_{i_2} \cdot a_{i_1, i_3}$. Then by the assumption $(g_1,1)\circ M$ respects $P$ we have $v_{i_2} \cdot a_{1,i_3} = 0$ for all $i_2 > k+1 - \lambda_{i_3}$ with $i_3 \in [k]$ or equivalently
$v_{i_2} \cdot a_{1,i_3} = 0$  for $1 \leq i_3 \leq \lambda_{k+2 - i_2}\transpose$ with $i_2 \in [k+1]$. However, since $\mu = \emptyset$ these are the only conditions that the collection of vectors $\{v_1, v_2, \ldots, v_{k+1}\}$ must satisfy in order for $(g_1,1) \circ M$ to respect $P$. 
Now we count how many choices of vectors $v_1, v_2, \ldots, v_{k+1}$ we have thereby counting the number of $g_1$.

For $i \in [k+1]$ define $S_{i} = \{ x \in \FF_q^{k+1} \mid x\cdot a_{1,j} = 0 \textnormal{ for all } 1 \leq j \leq \lambda_{k+2-i}\transpose\}$, that is $S_{i}$ is the set of all vectors $x \in \FF_{q}^{k+1}$ such that the dot product of $x$ with the first $\lambda_{k+2-i}\transpose$ column vectors of $M_1$ is zero. Notice that $S_{i}$ has the property that if $\{x_1, x_2, \ldots, x_n\} \subseteq S_{i}$ then $\Span{x_1, x_2, \ldots, x_n} \subseteq S_{i}$. By the fact $\{a_{1,1,}, a_{1,2}, \ldots, a_{1,k}\}$ is a linearly independent set of vectors we can use Lemma~\ref{lem:orthogonal} and calculate $|S_{i}| = q^{k+1 - \lambda_{k+2-i}\transpose}$.
Notice $\lambda_{k+2 - i}\transpose \leq \lambda_{k+2 - (i+1)}\transpose$, so $S_{k+1} \subseteq S_{k} \subseteq \cdots \subseteq S_{1}$. 
By assumption $(g_1,1) \circ M$ respects $P$ we have $v_{k+1} \in S_{k+1}$ and since $v_{k+1} \ne 0$ there are $q^{k+1 - \lambda_{1}\transpose} -1$ choices for $v_{k+1}$. 
Generally, for $1 \leq i \leq k$ we know $v_{k+1 - i} \in S_{k+1 - i}$ and $\Span{v_{k+1}, v_{k}, \ldots, v_{k+1 - (i-1)}} \subseteq S_{k+1 - i}$. By the linear independence of $\{v_{k+1}, \ldots, v_{k+1 - i}\}$ we have $v_{k+1 - i} \in S_{k+1 - i} \setminus \Span{v_{k+1}, v_{k}, \ldots, v_{k+1 - (i-1)}}$ so there are $q^{k+1 - \lambda_{i+1}\transpose} - q^i$ choices for $v_{k+1 - i}$. 
Therefore the number of choices of $v_{k+1}, v_k, \ldots, v_1$ is
\[(q^{k+1 - \lambda_1\transpose} - 1)(q^{k+1 - \lambda_2\transpose} -q)\cdots (q^{k+1 - \lambda_{k+1}\transpose} - q^k)\] which is also the number of choices $g_1$ such that $(g_1,1)\circ M$ respects $P$. This factors as \[q^{\binom{k+1}{2}}(q+1)^{k+1}[k+1 - \lambda_1\transpose]_q[k - \lambda_2\transpose]_q \cdots [2 - \lambda_{k}\transpose]_q[1 - \lambda_{k+1}\transpose]_q\] and by Lemma~\ref{lem:partitionTransposition} this is equal to 
\[q^{\binom{k+1}{2}}(q+1)^{k+1}[k+1 - \lambda_1]_q[k - \lambda_2]_q \cdots [2 - \lambda_{k}]_q.\]
Since we assumed $g_2 \in \GL_k(\FF_q)$ was arbitrary there are 
\[q^{k^2}(q-1)^{2k+1} [k+1 - \lambda_1]_q[k - \lambda_2]_q\cdots [2 - \lambda_k]_q \cdot [k]!_q\] pairs $(g_1,g_2) \in \GL_{k+1}(\FF_q) \times \GL_k(\FF_q)$ such that $(g_1,g_2) \cdot E$ respects $P$. 
Similar to the counting in Corollary~\ref{cor:aitkenCorollary} we conclude there are \[q^{k^2}(q-1)^{2k} [k+1 - \lambda_1]_q[k - \lambda_2]_q\cdots [2 - \lambda_k]_q \cdot [k]!_q\] nondegenerate hypermatrices that respect $P$.
\end{proof}

\subsection{The case \texorpdfstring{$P = \Delta$}{P equals Delta}}

We give one more family of plane partitions for which the conjecture holds.

\begin{thm}
The number of nondegenerate hypermatrices respecting $\Delta$ is 
\[q^{k^2}(q-1)^{2k}.\]
\end{thm}
\begin{proof}
Let $L_{k}, U_{k + 1}$ respectively denote the sets of lower-triangular $k \times k$ and upper-triangular $(k + 1) \times (k + 1)$ invertible matrices, let $S_1$ denote the set of nondegenerate $2 \times (k + 1) \times k$ hypermatrices that respect $\Delta$, and let $S_2$ denote the set of $2 \times (k + 1) \times k$ hypermatrices $M$ for which $M_{i, j, k} \neq 0$ if $j = i + k - 1$ and for which $M_{i, j, k} = 0$ if $j > i + k - 1$.  We claim that
\[
(U_{k + 1}, L_k) \circ E = S_1 = S_2.
\]
We proceed by proving two containments and then showing the largest and smallest of the three sets have the same size.

It is easy to see from the definition of the action of matrices on hypermatrices that if $A$ is upper-triangular and $B$ is lower-triangular then $(A, B) \circ E$ respects $\Delta$.  Since $E$ is nondegenerate, it follows immediately from Proposition~\ref{prop:invariant} that $(U_{k + 1}, L_k) \circ E \subseteq S_1$.  By  Theorem~\ref{thm:MaxPlanePartition}, if a hypermatrix $M$ respects $\Delta$ and $M_{1,i_3,i_3} = 0$ or $M_{2,i_3+1,i_3} = 0$ for $i_3 \in [k]$ then $M$ is degenerate; consequently $S_1 \subseteq S_2$.  Finally, by directly counting the number of choices for each entry we observe that $|S_2| = q^{\binom{k}{2} + \binom{k + 1}{2}}(q - 1)^{2k} = q^{k^2}(q-1)^{2k}$, whereas since both $U_{k + 1}$ and $L_k$ contain the scalar matrices, we have by Theorem~\ref{thm:aitken} that $|(U_{k + 1}, L_k) \circ E| = \frac{1}{q - 1}|U_{k + 1}| \cdot |L_k| = q^{k^2}(q-1)^{2k}$, as well.  Thus in fact all three sets are equal, and have the claimed cardinality.
\end{proof}

\subsection{The case \texorpdfstring{$k = 2$}{k = 2}}

Although we do not have a general theorem to count the number of nondegenerate hypermatrices for an arbitrary plane partition with $\mu \ne \emptyset$ many of the methods we have developed so far can be applied to compute concrete cases. To illustrate the use of these techniques we will compute the number of nondegenerate hypermatrices in the $2 \times 3 \times 2$ case with $P \ne \Delta$ and $\mu \ne \emptyset$. In this case there are three plane partitions which satisfy these properties:
\begin{enumerate}
\item{ $P_1$ with $\lambda = \brac{1,0}$ and $\mu = \brac{1,0}$,}
\item{$P_2$ with $\lambda = \brac{2,0}$ and $\mu = \brac{1,0}$,}
\item{ $P_3$ with $\lambda = \brac{1,1}$ and $\mu = \brac{1,0}$.}
\end{enumerate}  Although this case is small the general strategies used to compute the number of nondegenerate hypermatrices that respect $P_i$ for $i = 1,2,3$ can be used in larger cases, especially the method used for $i = 3$. Also, the number of plane partitions with $P \ne \Delta$ and $\mu \ne \emptyset $ grows rapidly, for example in the $2 \times 4 \times 3$ case there are $40$ such $P$ and we would rather not put the reader through all of that. 
 
Let us start with $P_1$ and look at the hypermatrix $(1,g_2) \circ E$ for some $g_2 \in \GL_2(
\FF_q)$. Let 
$g_2 = \begin{bmatrix} a&b\\c&d\end{bmatrix}$ and recall there are $q(q-1)^2[2]_q$ such $g_2$. By definition of the action
\[(1,g_2) \cdot E = \begin{bmatrix} a&c\\b&d\\0&0\end{bmatrix}\begin{bmatrix} 0&0\\a&c\\b&d\end{bmatrix}\] and recall by Proposition~\ref{prop:invariant} this hypermatrix is nondegenerate.
Therefore by Lemma~\ref{lem:faceSum} the vectors $\begin{bmatrix}a\\b\\0\end{bmatrix}$ and $\begin{bmatrix}0\\a\\b\end{bmatrix}$ are linearly independent. 
Now let us take some $g_1 \in \GL_3(\FF_q)$ such that $(g_1, g_2) \circ E$ respects $P_1$.
We will think of $g_1$ as the matrix of row vectors 
$\begin{bmatrix} \text{---} \hspace{-0.2cm} & r_1 & \hspace{-0.2cm} \text{---} \\ \text{---} \hspace{-0.2cm} & r_2 & \hspace{-0.2cm} \text{---} \\ \text{---} \hspace{-0.2cm} & r_3 & \hspace{-0.2cm} \text{---} \end{bmatrix}$.
By assumption $(g_1, g_2) \circ E$ respects $P$ we have $r_3 \cdot \begin{bmatrix}a\\b\\0\end{bmatrix} = r_3 \cdot \begin{bmatrix}0\\a\\b\end{bmatrix} = 0$ and by Lemma~\ref{lem:orthogonal} there are $q-1$ choices for $r_3$. The second row, $r_2$, has no restrictions other than it must be linearly independent from $r_3$, so there are $q^3 - q$ choices for $r_2$. Similarly, $r_1$ has no restrictions other than being linearly independent from $r_2$ and $r_3$ so there are $q^3 - q^2$ choices of $r_1$. Thus the number of $(g_1,g_2)$ such that $(g_1, g_2) \cdot E$ respects $P_1$ is 
\[(q-1)(q^3 -q)(q^3 - q^2)\cdot q(q-1)^2[2]_q = q^4(q-1)^5[2]_q^2.\] Therefore the number of nondegenerate hypermatrices which respect $P_1$ is $q^4(q-1)^4[2]_q^2$.

We play the same game on $P_2$, choosing any $g_2 \in \GL_2(\FF_q)$ which we represent with the same matrix as before. 
Again we let $g_1 \in \GL_3(\FF_q)$ such that $(g_1, g_2)\circ E$ respects $P_2$ and we represent $g_1$ as the matrix of row vectors above. As before we must have $r_3 \cdot \begin{bmatrix}a\\b\\0\end{bmatrix} = 0$ and $r_3 \cdot \begin{bmatrix}0\\a\\b\end{bmatrix} = 0$ so there are $q-1$ choices for $r_3$.
Now we must have $r_2 \cdot \begin{bmatrix}a\\b\\0\end{bmatrix} = 0$ and by assumption $g_1$ is invertible $r_2 \notin \Span{r_3}$. However every scalar multiple of $r_3$ is also a choice for $r_2$ since \[(cr_3) \cdot \begin{bmatrix}a\\b\\0\end{bmatrix} = c\left(r _3\cdot\begin{bmatrix}a\\b\\0\end{bmatrix}\right) = 0\] and so there are $q^2 - q$ choices for $r_2$. 
Now we need to choose $r_1$ and the only restriction we have comes from the assumption $g_1$ is invertible so $r_1 \notin \Span{r_3,r_2}$ and there are $q^3 - q^2$ choices for $r_1$. Thus the number of $(g_1,g_2) \in \GL_3(\FF_q) \times \GL_2(\FF_q)$ such that $(g_1,g_2) \circ E$ respects $P_2$ is 
\[(q-1)(q^2-q)(q^3-q^2)q(q-1)^2[2]_q = q^4(q-1)^5[2]_q\] and the number of nondegenerate hypermatrices that respect $P_2$ is $q^4(q-1)^4[2]_q$.

In the previous two cases, we were able to get by with the same techniques as in the case $\mu = \emptyset$ of Theorem~\ref{thm:muEmpty}. For $P_3$ something genuinely new happens; that is for every choice of $g_2 \in \GL_2(\FF_q)$ it is not true that there are the same number of $g_1\in \GL_3(\FF_q)$ such that $(g_1,g_2) \circ E$ respects $P_3$. Instead of assuming $g_2 \in \GL_2(\FF_q)$ is arbitrary we will consider the following two cases: $g_2 = \begin{bmatrix} a&0\\c&d\end{bmatrix}$ and $g_2 = \begin{bmatrix} a&x\\c&d\end{bmatrix}$ where $x \in \FF_q^{\times}$. In the first case 
\[(1,g_2) \circ E = \begin{bmatrix} a&c\\ 0&d\\0&0\end{bmatrix}\begin{bmatrix} 0&0
\\a&c\\ 0&d\end{bmatrix} \] and now let us count the $g_1 \in \GL_3(\FF_q)$ such that $(g_1,g_2) \circ E$ respects $P_3$. Recall the diagonal of an invertible upper triangular matrix is nonzero and since $g_2$ is invertible, $a,d \ne 0$ and $c$ is free so there are $q(q-1)^2$ such $g_2$.
Let the third row of $g_1$ be $\begin{bmatrix} w&y&z\end{bmatrix}$. We must have $\begin{bmatrix} w&y&z\end{bmatrix} \cdot \begin{bmatrix} a\\0\\0 \end{bmatrix} = 0$ so $w = 0$ and similarly $\begin{bmatrix} w&y&z\end{bmatrix} \cdot \begin{bmatrix} 0\\a\\0 \end{bmatrix} = 0$ so $y = 0$. 
By assumption $g_1$ is invertible we know $z \ne 0$ so $z \in \FF_q^{\times}$. For any such $z$,
\[ \begin{bmatrix}0&0&z\end{bmatrix} \cdot\begin{bmatrix} a\\0\\0 \end{bmatrix} = \begin{bmatrix}0&0&z\end{bmatrix} \cdot\begin{bmatrix} 0\\a\\0\end{bmatrix} = \begin{bmatrix}0&0&z\end{bmatrix} \cdot\begin{bmatrix} c\\d\\0 \end{bmatrix} = 0\] and thus there are $q-1$ choices for $\begin{bmatrix} w&y&z\end{bmatrix}$. Since those are all the restrictions on $g_1$ the number of such matrices is $(q-1)(q^3-q)(q^3-q^2)$.

Let us see what happens when we let $g_2 = \begin{bmatrix} a&x\\c&d\end{bmatrix}$ with $x \in \FF_q^\times$. By definition of the action we have 
\[(1,g_2) \cdot E = \begin{bmatrix} a&c\\x&d\\0&0\end{bmatrix}\begin{bmatrix} 0&0
\\a&c\\x&d\end{bmatrix} \] and let us assume there exists some $g_1 \in \GL_3$ such that $(g_1,g_2) \circ E$ respects $P_3$.
Let us represent $g_1$ as the matrix of row vectors $\begin{bmatrix} \text{---} \hspace{-0.2cm} & r_1 & \hspace{-0.2cm} \text{---} \\ \text{---} \hspace{-0.2cm} & r_2 & \hspace{-0.2cm} \text{---} \\ \text{---} \hspace{-0.2cm} & r_3 & \hspace{-0.2cm} \text{---} \end{bmatrix}$. 
By the assumption that $(g_1,g_2)\circ E$ respects $P_3$ we have $r_3 \cdot \begin{bmatrix} a\\x\\0 \end{bmatrix} = r_3 \cdot \begin{bmatrix} c\\d\\0 \end{bmatrix} =0$. Since $\begin{bmatrix} a\\x\\0 \end{bmatrix}$ and $\begin{bmatrix} c\\d\\0 \end{bmatrix}$ are linearly independent by Lemma~\ref{lem:orthogonal} there are $q-1$ possible choices for $r_3$. We can see these are the vectors $\begin{bmatrix}0&0&z\end{bmatrix}$ with $z\in \FF^{\times}_q$.
However, by assumption $(g_1,g_2)\circ E$ respects $P_3$ we also have 
\[r_3 \cdot \begin{bmatrix}0\\a\\x\end{bmatrix} = zx = 0\] which is a contradiction since we assumed both $z,x \in \FF_q^{\times}.$ Thus for any choice of $g_2 = \begin{bmatrix}a&x\\c&d\end{bmatrix}$ with $x \in \FF_q^{\times}$ there are no $(g_1,g_2) \in \GL_3(\FF_q) \times \GL_2(\FF_q)$ such that $(g_1,g_2) \circ E$ respects $P_3$. Thus the number of elements $(g_1,g_2)$ such that $(g_1,g_2) \circ E$ respects $P_3$ is 
\[(q-1)(q^3-q)(q^3-q^2)q(q-1)^2 = q^4(q-1)^5[2]_q\] and the number of nondegenerate hypermatrices that respect $P_3$ is $q^4(q-1)^4[2]_q$.

\section{Hyperrook placements and a weak form of the main conjecture}\label{sec:4}

We now switch gears and look at the ``moral" evidence supporting the conjecture. We begin by defining a natural generalization of rook placements that we call \emph{hyperrook placements}.  The main result of this section is the following theorem, which verifies several nontrivial predictions of Conjecture~\ref{mainConjecture}, namely, that the number of nondegenerate $2 \times (k + 1) \times k$ hypermatrices over $\FF_q$ with support on a plane partition is a polynomial function of $q$ that has (up to a power of $q - 1$) nonnegative integer coefficients, and furthermore provides a combinatorial interpretation for the value of this polynomial at $q = 1$.

\begin{restatable}{thm}{weaktheorem}
\label{thm:weak form of conjecture}
    Let $P \subseteq \Delta$.  There is a polynomial $f$ with nonnegative integer coefficients such that
    \[
    f(1) = (k+1 - \lambda_1)(k - \lambda_2) \cdots (2 - \lambda_k)\cdot(k-\mu_1)((k-1) - \mu_2) \cdots (1 - \mu_k)
    \]
    is the number of hyperrook placements that respect $P$ and, for every prime power $q$, the number of nondegenerate hypermatrices that respect $P$ is
    \[
    q^{k^2} (q - 1)^{2k} \cdot f(q).
    \]
\end{restatable}

\subsection{Hyperrook placements}\label{sec:hyperrooks}

We begin by defining the combinatorial model we use for higher-dimensional rook placements.

\begin{definition}\label{def:hyperrook placement}
Let $\sigma \in \Symm_{k+1}$ and $\pi \in \Symm_k$ and $m_{\sigma}$, $m_{\pi}$ be their corresponding permutation matrices in $\GL_{k+1}(\FF)$ and $\GL_{k}(\FF)$, respectively. A $2 \times (k+1) \times k$ \defn{hyperrook placement} is any hypermatrix of the form
\[(m_{\sigma}, m_\pi) \circ E\] where $E$ is the hypermatrix defined in Definition~\ref{E}.
\end{definition}

\begin{example}\label{eg:a hyperrook placement}
The hyperrook placement $(m_{4213}, m_{312}) \circ E$ is given by
\begin{multline*}
(m_{4213}, m_{312}) \circ E = (m_{4213}, 1) \circ \left(1, \begin{bmatrix}
0&1&0\\
0&0&1\\
1&0&0
\end{bmatrix} \right) \circ \left(
\begin{bmatrix}
1&0&0\\
0&1&0\\
0&0&1\\
0&0&0
\end{bmatrix}
\begin{bmatrix}
0&0&0\\
1&0&0\\
0&1&0\\
0&0&1
\end{bmatrix}\right)
= {}\\
\left(\begin{bmatrix}
0&0&1&0\\
0&1&0&0\\
0&0&0&1\\
1&0&0&0
\end{bmatrix}, I\right)
\circ \left(
\begin{bmatrix}
0&0&1\\
1&0&0\\
0&1&0\\
0&0&0
\end{bmatrix}
\begin{bmatrix}
0&0&0\\
0&0&1\\
1&0&0\\
0&1&0
\end{bmatrix} \right)
 =
\begin{bmatrix}
0&1&0\\
1&0&0\\
0&0&0\\
0&0&1
\end{bmatrix}
\begin{bmatrix}
1&0&0\\
0&0&1\\
0&1&0\\
0&0&0
\end{bmatrix}.
\end{multline*}
\end{example}

The question of whether there is a more intrinsic definition of hyperrook placements is raised in Section~\ref{sec:intrinsic} below.

We now build up to showing that the number of hyperrook placements that respect a plane partition $P \subseteq \Delta$ corresponds to the conjectured number of nondegenerate hypermatrices that respect $P$ as in Conjecture~\ref{mainConjecture}.  In order to do this, we need to establish some definitions.
\begin{definition}\label{wc}
Given $(\sigma, \pi) \in \Symm_{k+1} \times \Symm_k$, let $\ol{\pi}$ be the permutation in $\Symm_{k+1}$ with $\ol{\pi}(i) = \pi(i)$ for $ i \in [k]$ and $\ol{\pi}(k+1) = k+1$. Define $w,c \in \Symm_{k+1}$ by $w = \sigma \cdot \ol{\pi}$ and $c = \ol{\pi}\inv(1\ 2\ \cdots \ k\ k+1)\ol{\pi}$.
\end{definition} 
 
 \begin{definition}
    Let $\CCC_{k+1} \subseteq \Symm_{k+1} \times \Symm_{k+1}$ be the set of all ordered pairs $(\alpha,\beta)$ where $\alpha$ is an arbitrary element of $\Symm_{k+1}$ and $\beta$ is a $(k+1)$-cycle.
\end{definition}
It turns out that the map of Definition~\ref{wc} is actually a bijection, as the following proposition shows.
\begin{prop}\label{prop:curly C}
The function $\Symm_{k+1} \times \Symm_k \to \CCC_{k+1}$ that maps $(\sigma, \pi) \mapsto (w, c)$ as in Definition~\ref{wc} is a bijection.
\end{prop}
\begin{proof}
First we show the map is injective. 
For $(\sigma_1, \pi_1), (\sigma_2, \pi_2) \in \Symm_{k+1} \times \Symm_k$, suppose
\[\ol{\pi_1}\inv(1\ 2\ \cdots\ k+1)\ol{\pi_1} = \ol{\pi_2}\inv(1\ 2\ \cdots\ k+1)\ol{\pi_2}\] and $\sigma_1 \cdot \ol{\pi_1} = \sigma_2 \cdot \ol{\pi_2}$. Then 
\[\ol{\pi_2}\,\ol{\pi_1}\inv(1\ 2\ \cdots\ k+1)(\ol{\pi_2}\,\ol{\pi_1}\inv)\inv = (1\ 2\ \cdots k+1).\] Therefore
\[\big( \ol{\pi_2}\,\ol{\pi_1}\inv(1) \ \ \ol{\pi_2}\,\ol{\pi_1}\inv(2) \ \ \cdots \ \ \ol{\pi_2}\,\ol{\pi_1}\inv(k) \ \ \ol{\pi_2}\inv \ol{\pi_1}(k+1) \big) = (1\ 2\ \cdots\ k \ k+1).\] Notice $\ol{\pi_1}\inv(k+1) = \ol{\pi_2}(k+1) = k+1$ so $\ol{\pi_2}\,\ol{\pi_1}\inv(k+1) = k+1$.  Therefore we must have $\ol{\pi_2}\ol{\pi_1}\inv(i) = i$ for all $i \in [k]$.  But then $\pi_1(i) = \ol{\pi_1}(i) = \ol{\pi_2}(i) = \pi_2(i)$ for all $i \in [k]$, and so $\pi_1 = \pi_2$. Consequently, $\ol{\pi_1} = \ol{\pi_2}$ and by the assumption $\sigma_1 \cdot \ol{\pi_1} = \sigma_2 \cdot \ol{\pi_2}$ we get $\sigma_1 = \sigma_2$. Therefore $(\sigma_1, \pi_1) = (\sigma_2, \pi_2)$ and the map is injective.
Since the domain and codomain both have size $(k + 1)! \cdot k!$, the map is a bijection, as claimed.
\end{proof}

Consider the pair $(\sigma, \pi) = (4213, 231)$.  We compute $w = 4213 \cdot 2314 = 2143,$ $c = 3124 \cdot (1\ 2\ 3 \ 4) \cdot 2314 = (3\ 1 \ 2 \ 4)$, and $wc = 2143 \cdot (3\ 1\ 2\ 4) = 1324$. Comparing this to the hyperrook placement $(m_{4213}, m_{312}) \circ E$ in Example~\ref{eg:a hyperrook placement}, we see the $1$s are exactly in positions $(1,w(j), j)$ and $(2, wc(j), j)$ for $j = 1,2,3$. The next lemma shows this is not a coincidence. 

\begin{lemma}\label{lem:hyperrookWC}
    Let $(\sigma, \pi) \in \Symm_{k+1} \times \Symm_{k}$ and $(w,c)$ be defined as above. Let $M$ be the hypermatrix of format $2 \times (k+1) \times k$ with $M_{1,i,j} = 1$ if $i = w(j)$ and $M_{2,i,j} = 1$ if $i = wc(j)$ for $j \in  [k]$, and all other entries zero. Then $(m_{\sigma}, m_{\pi\inv}) \circ E = M$.
\end{lemma}
\begin{proof}
Denote by $A_1$, \ldots, $A_k$ the slices of $E = (e_{i_1, i_2, i_3})$ in the third direction and by $B_1, \ldots, B_k$ the slices of $(1, m_{\pi\inv}) \circ E$ in the third direction.  By the definition of the group action, we have for $n = 1, \ldots, k$ that
\[
B_n = (m_{\pi\inv})_{n,1} \cdot A_{1} + (m_{\pi\inv})_{n,2}\cdot A_{2} + \cdots  + (m_{\pi\inv})_{n,k} \cdot A_{k} = A_{\pi(n)}.
\]
Thus, the $(i_1, i_2, i_3)$-entry of $(1, m_{\pi\inv}) \circ E$ is $e_{i_1, i_2, \pi(i_3)}$.  Considering in the same way the action of $(m_{\sigma}, 1)$ on $(1, m_{\pi\inv}) \circ E$ by permuting slices in the second direction, we conclude that the $(i_1, i_2, i_3)$-entry of $(m_{\sigma},m_{\pi\inv})\circ E$ is $e_{i_1, \sigma^{-1}(i_2), \pi(i_3)}$.  Since $e_{1, i, j}$ is $1$ if $i = j$ and is $0$ otherwise, and $e_{2, i, j}$ is $1$ if $i = j + 1$ and is $0$ otherwise, we conclude that the $(1, i_2, i_3)$-entry of $(m_{\sigma},m_{\pi\inv})\circ E$ is $1$ when $\pi(i_3) = \sigma^{-1}(i_2)$ and the $(2, i_2, i_3)$-entry is $1$ when $\pi(i_3) + 1= \sigma^{-1}(i_2)$.
Equivalently, the $(1,i_2,i_3)$ entry is  $1$ when \[w(i_3) = \sigma ( \ol{\pi}(i_3)) = i_2\] for $1 \leq i_3 \leq k$, and the $(2,i_2,i_3)$-entry is $1$ when $\sigma(\ol{\pi}(i_3) + 1) = i_2$ for $1 \leq i_3 \leq k$. Notice $\ol{\pi}(i_3) + 1 = (1\ 2\ \cdots \ k+1)\ol{\pi}(i_3)$ for all $1 \leq i_3 \leq k$. Therefore 
\[\sigma(\ol{\pi}(i_3) + 1) = \sigma(1\ 2\ \cdots \ k+1)\ol{\pi}(i_3) = \sigma \ol{\pi} \ol{\pi}\inv (1\ 2\ \cdots \ k+1)\ol{\pi}(i_3) = wc(i_3)\] for all $1 \leq i_3 \leq k$. Thus the $(2, i_2, i_3)$-entry of $(m_\sigma, m_{\pi\inv})\circ E$ is $1$ precisely when $wc(i_3) = i_2$ for $1 \leq i_3 \leq k$. Since all other entries of $(m_{\sigma},m_{\pi\inv}) \circ E$ are $0$, we conclude that $(m_\sigma, m_{\pi\inv}) \circ E = M$, as claimed.
\end{proof}

\begin{cor}
\label{lem:wc restrictions}
Let $P \subseteq \Delta$ be a plane partition of format $2 \times (k+1) \times k$. A hyperrook placement $(m_{\sigma}, m_{\pi\inv}) \circ E$ respects $P$ if and only if the following two conditions are satisfied for all $1 \leq j \leq k$:
\begin{enumerate}
\item{ $w(j) \leq k+1 - \lambda_j$} and
\item{$wc(j) \leq k+1 - \mu_j$},
\end{enumerate}
where $w$ and $c$ are as in Definition~\ref{wc}.
\end{cor}

\begin{proof}
    Let $(m_{\sigma}, m_{\pi\inv})\circ E$ be a hyperrook placement. By definition, $M := (m_{\sigma}, m_{\pi\inv})\circ E$ respects $P$ if and only if $M_{1,i,j} = 0$ for all $i > k+1 - \lambda_j$ and $M_{2,i,j} = 0$ for all $i > k+1 - \mu_j$.   By Lemma~\ref{lem:hyperrookWC}, we have for all $1 \leq j \leq k$ that $M_{1,i,j} = 1$ if $i = w(j)$ and $M_{2,i,j} = 1$ if $i = wc(j)$, and $M_{h,i,j} = 0$, otherwise.   Thus $M$ respects $P$ if and only if the indices of the nonzero entries $M_{1, w(j), j}$ and $M_{2, wc(j), j}$ do \emph{not} satisfy the inequalities $w(j) > k + 1 - \lambda_j$ or $wc(j) > k + 1 - \mu_j$, and this is precisely the claim.
\end{proof}

Using the previous fact we can now show the number of hyperrook placements that respect $P \subseteq \Delta$ corresponds to the conjectured number of nondegenerate hypermatrices that respect $P$ at $q = 1$.

\begin{thm}\label{thm:hyperrookCount}
Fix a format $2 \times (k+1) \times k$ and let $P \subseteq \Delta$ be a plane partition. There are 
\[(k+1 - \lambda_1)\cdots (2 - \lambda_k) \cdot (k- \mu_1) \cdots (2 - \mu_{k-1})(1-\mu_k)\] elements $(w,c) \in \CCC_{k+1}$ such that $w(i) \leq k+1 - \lambda_i$ and $wc(i) \leq k+1-\mu_i$ for $1\leq i \leq k$.
\end{thm}

\begin{remark}
    It's not hard to see that if $P$ is a plane partition such that $P \not\subseteq \Delta$, then the given product is $0$.
\end{remark}

\begin{proof}
Let $P \subseteq \Delta$ and $w \in \Symm_{k+1}$ with $w(i) \leq k+1 - \lambda_i$ for $i \in [k]$. We will first show the number of $(k+1)$-cycles $c \in \Symm_{k+1}$ such that $wc(i) \leq k+1 - \mu_i$ for $i \in [k]$ is \[(k-\mu_1)\cdot (k-1-\mu_2) \cdots (1-\mu_k),\] independent of the choice of $w$. The general strategy we will use to prove this claim is to induct on $k$, so we need some way to relate these sets to a previous case.
Since $w \in \Symm_{k+1}$, there are $k+1 - \mu_i$ integers $j \in [k+1]$ such that $w(j) \leq k+1 - \mu_i$. Also, since $w(i) \leq k+1 - \lambda_i$ and $\mu_i \leq \lambda_i$, we know $w(i) \leq k+1 - \mu_i$ and particular $w(1) \leq k+1 - \mu_1$. For a $(k+1)$-cycle $c \in \Symm_{k+1}$ we know $c(i) \ne i$, so if $c$ is some $(k+1)$-cycle such that $wc(i) \leq k+1 - \mu_i$ then $c(1) \ne 1$. If $a_1, a_2, \ldots, a_{k - \mu_1} \in [k+1]$ are the $k - \mu_1$ integers with $a_j \ne 1$ and \[w(a_j) \leq k+1 - \mu_1,\] 
then we know $c(1) = a_j$ for some $j = 1,2, \ldots, k - \mu_1$, so there are $(k - \mu_1)$ choices of $c(1)$. 
We now define the sets $C_{1}, C_{2}, \ldots, C_{{k- \mu_1}}$ where 
$$C_{j} : = \{ c : c(1) = a_j \textnormal{ and } wc(i) \leq k+1 - \mu_i \textnormal{ for all } i \in [k]\}.$$
We will show that each of these sets has the same size, by giving bijections between them and a set coming from a plane partition in a smaller box.  To this end, we make some more definitions.

Let $P' = \brac{ \lambda', \mu'}$ be the plane partition of format $2 \times k \times (k-1)$ where $\lambda_i' = \lambda_{i+1}$ and $\mu_i' = \mu_{i+1}$ for $i \in [k-1]$. Since $P \subseteq \Delta = \Delta_{1,k}$ we have $\lambda_i \leq k+1 - i$ and $\mu_i \leq k - i$. Therefore, by definition $\lambda_i' \leq k+1 - (i+1) = k - i$ and $\mu_i' \leq k - (i+1) = (k-1) - i$ so $P' \subseteq \Delta_{1, k-1}$. We now define $w' \in \Symm_k$ by 
\[w'(i) = \begin{cases} w(i+1) & \textnormal{ if }w(i+1) < w(1)\\ w(i+1)-1 & \textnormal{ if } w(i+1) > w(1) \end{cases}\] and we claim $w'(i) \leq k - \lambda_i'$ for $i \in [k-1]$.
We have two cases: 
if $w'(i) = w(i+1)$ then 
\[w'(i) < w(1) \leq k+1 - \lambda_1 \leq k + 1 - \lambda_{i+1} = k+1 - \lambda_i'\] and if $w'(i) = w(i + 1) - 1$ then
$$w'(i) \leq (k+1 - \lambda_{i+1}) - 1 = k - \lambda_{i+1} = k - \lambda_{i}'.$$ Therefore $w'(i) \leq k - \lambda'_i$ for all $i \in [k-1]$ as claimed. We now define $B$ to be the set of $k$-cycles $c' \in \Symm_k$ such that $w'c'(i) \leq k - \mu'_i$ for all $i \in [k - 1]$

We construct a collection of bijections $f_{1}, f_{2}, \ldots, f_{{k - \mu_1}}$ where $$f_{j} : C_{j} \to B.$$ For $c \in C_{j}$ with $c = (1 \ c(1) \ \cdots \ c^k(1) )$ we define $f_{j}$ by 
$$f_{j}(c) = \left((c(1) - 1) \ \ (c^2(1) - 1)\ \ \cdots \ \ (c^k(1) - 1) \right).$$
This is clearly a $k$-cycle in $\Symm_k$, but for $f_{j}$ to be well defined we must check that $w' (f_{j}(c)(i)) \leq k - \mu'_i$ for $i \in [k-1]$. To do this we consider two cases on $i \in [k-1]$. First, if $i \ne c^k(1) - 1$, then by definition $f_{j}(c)(i) = c(i+1) - 1$. Therefore $w' (f_{j}(c)(i)) = w'(c(i + 1) - 1)$ so we have the following two possibilities:
$$w' (f_{j}(c)(i)) = \begin{cases} w(c(i+1)) &\textnormal{ if }w(c(i+1)) < w(1)\\ w(c(i+1))-1 &\textnormal{ if } w(c(i+1)) > w(1) \end{cases}.$$ 
If $w' (f_{j}(c)(i)) = w(c(i+1))$ then
$w(c(i+1)) < w(1) \leq k+1 - \mu_1$ and therefore 
$$w' (f_{j}(c)(i)) \leq k - \mu_1 \leq k - \mu_{i+1} = k - \mu'_i.$$ 
If $w' (f_{j}(c)(i)) = w(c(i+1)) - 1$ then by the definition of $C_j$ we have $w(c(i+1)) \leq k+1 - \mu_{i+1}$ and so
$$w' (f_{j}(c)(i)) = w(c(i+1)) - 1 \leq k - \mu_{i+1} = k - \mu'_{i}.$$
Therefore if $i \ne c^k(1) - 1$ we conclude that $w'(f_{j}(c)(i)) \leq k - \mu'_i$.
Now we prove the case when $i = c^k(1) - 1$. By definition $f_{j}(c)(c^k(1) - 1) = c(1) - 1$ and again we have the two cases. First, if $w'(c(1) - 1) = w(c(1)) < w(1)$ then since $w(1) < k+1  - \mu_1$ we have
$$w'(c(1) - 1) \leq w(1) - 1 \leq k - \mu_1 .$$ If $w'(c(1) - 1) = w(c(1)) - 1$ we have 
$$w'(c(1) - 1) \leq (k+1 - \mu_1) -1 = k - \mu_1,$$ so either way so $w' (f_{j}(c)(c^k(1) - 1)) \leq k - \mu_1 \leq k - \mu_{c^k(1)} = k - \mu'_{c^k(1) -1}$. Thus, $w'(f_{j}(c)(i)) \leq k - \mu_{i}'$ for all $i \in [k-1]$ so $f_{j}$ is well defined for $j = 1, 2, \ldots, k- \mu_1$. 

Next we show each $f_{j}$ is a bijection by constructing a family of inverse functions $g_{j}: B \to C_{j}$ for $j = 1,2, \ldots, k - \mu_1$. Let $c_* \in B$ which we can write as
$$c_*= (1 \ c_*(1) \ \cdots \ c_*^{k-1}(1)),$$ and since $a_j > 1$ there exists some $n >0$ such that $c_*^n(1) = a_j - 1$. Now define 
$$g_{j}(c_*) = \left(1\ a_j \ (c_*^{n+1}(1) + 1) \ \cdots \ (c_*^{n-1}(1) + 1) \right),$$ which is clearly a $(k+1)$-cycle in $\Symm_{k+1}$ with $g_{j}(c_*)(1) = a_j$. But, to show that $g_{j} \in C_{j}$ we still need to verify $w(g_{j}(c_*)(i)) \leq k+1 - \mu_i$ for $i \in [k]$. We again consider two cases on $i$.  If $i \ne c_*^{n-1}(1)$ then by definition $g_{j}(c_*)(i + 1) = c_*(i) + 1$. By the definition of $w'$ we know $w(i + 1) \leq w'(i) + 1$ and thus
$$w( g_{j}(c_*)(i+1)) = w(c_*(i) + 1) \leq w'c_*(i) + 1 \leq k+1 - \mu'_i = k + 1 - \mu_{i+1}.$$ Again by definition $g_{j}(c_*)(c_*^{n-1}(1) + 1) = 1$ so
$$w(g_{j}(c_*)(c_*^{n-1}(1) + 1)) = w(1) \leq k + 1 - \lambda_1 \leq k+1 - \mu_{c_*^{n-1}(1) + 1}.$$ Finally,
$w(g_{j}(c_*)(1)) = w(a_j)$ and by assumption $w(a_j) \leq k+1 - \mu_1$, as needed. Therefore $w(g_{j}(c_*)(i)) \leq k+1 - \mu_i$ for all $i \in [k]$ so $g_{j}$ is well defined for $j = 1, 2, \ldots, k - \mu_1$. 

Now we show $g_{j}$ is a left and right inverse of $f_{j}$. Let $c \in C_{j}$ with $c = (1\ a_j \ c^2(1) \ \cdots \ c^k(1))$, then 
$$g_{j} \circ f_{j}(c) = g_{j} \left((a_j -1) \ \ (c^2(1) - 1)\ \ \cdots \ \ (c^k(1) - 1) \right) = (1\ a_j \ c^2(1) \ \cdots \ c^k(1)).$$ For $c_* \in B$ with $c_* = (1 \ c_*(1) \  \cdots \ c_*^{k-1}(1))$ we have
\begin{align*}
    f_{j} \circ g_{j}(c_*) &=f_{j}\left(1\ a_j \ (c_*^{n+1}(1) + 1) \ \cdots \ (c_*^{n-1}(1) + 1) \right) \\
    &= f_{j}\left(1\ (c_*^n(1) + 1) \ (c_*^{n+1}(1) + 1) \ \cdots \ (c_*^{n-1}(1) + 1) \right)
    \end{align*}
    by assumption that $c_*^n(i) + 1 = a_j$. Then we can see 
$$f_{j}\left(1\ (c_*^n(1) + 1) \ (c_*^{n+1}(1) + 1) \ \cdots \ (c_*^{n-1}(1) + 1) \right) = ( c_*^n(1) \ c_*^{n+1}(1) \ \cdots \ c_*^{n-1}(1) ) = c_*$$ by the translation invariance of cycles. Therefore $f_{j}$ is a bijection which shows $|C_{j}| = |B|$. Furthermore, since any $c$ such that $wc(i) \leq k+1 - \mu_i$ is in exactly one set $C_{j}$ we conclude the number of such $c$ is 
$$\left| \bigsqcup_{j = 1}^{k - \mu_1} C_{j}\right| = \sum_{j = 1}^{k - \mu_1} |C_{j}| = (k- \mu_1)|B|.$$

Equipped with this previous fact we are ready to compute the number of $(k+1)$-cycles $c \in \Symm_{k+1}$ such that $wc(i) \leq k+1 - \lambda_i$ for some fixed $w$ via induction. For $2 \times (n+1) \times n$ hypermatrices we induct on $n$. For the base case $n = 1$, any plane partition must be contained in $\Delta_{1,1}$ where $\lambda_1 = 1$ and $\mu_1 = 0$. 
For any $w \in \Symm_2$ and any plane partition $P$ we always have $wc(1) \leq 2$ and $wc(2) \leq 2$, and so the single choice $c = (1\ 2)$ holds for all $w$. So the number of such $c$ is $1 - \mu_1 = 1$ as claimed. 
For $n = k-1$ assume that for any plane partition $P \subseteq \Delta_{1,k-1}$ and $w \in \Symm_k$ with $w(i) \leq k - \lambda_i$ for $i \in [k-1]$ that the number of $k$-cycles $c \in \Symm_{k}$ such that $wc(i) \leq k - \mu_i$ for $i \in [k-1]$, is
$$((k-1) - \mu_1)( (k-2) - \mu_2) \cdots (1 - \mu_{k-1}).$$
Now let $P \subseteq \Delta_{1,k}$ and $w \in \Symm_{k+1}$ with $w(i) \leq k+1 - \lambda_i$ for $i \in [k+1]$. By what we have just shown the number of $c$ such that $wc(i) \leq k+1 - \mu_i$ is $(k - \mu_1)|B|$. By the inductive hypothesis 
$$|B| = ((k-1) - \mu_2)( (k-2) - \mu_3) \cdots (1 - \mu_{k}),$$ so the number of $c$ is 
$$(k - \mu_1)((k-1) - \mu_2)( (k-2) - \mu_3) \cdots (1 - \mu_{k}),$$ as claimed.

Now we count the pairs $(w,c) \in \CCC_{k+1}$ that respect some $P \subseteq \Delta$. Let $w \in \Symm_{k+1}$ such that $w(i) \leq k+1 - \lambda_i$ for all $i \in [k+1]$. It is straightforward that the number of $w$ is 
$$(k+1 - \lambda_1)(k - \lambda_2) \cdots (2 - \lambda_k)$$ and for any $w$ we have the same number of $(k+1)$-cycles $c$ such that $wc(i) \leq k+1 - \mu_i$ for all $i \in [k+1]$. From our counting above we conclude the number of pairs $(w,c)\in \CCC_{k+1}$ that respect $P$ is 
$$(k+1 - \lambda_1)(k - \lambda_2) \cdots (2 - \lambda_k)\cdot(k-\mu_1)((k-1) - \mu_2) \cdots (1 - \mu_k),$$ as claimed.
\end{proof}

\subsection{Decomposing the space of nondegenerate hypermatrices}
\label{sec:Bruhat}

The rest of this section is devoted to the proof of Theorem~\ref{thm:weak form of conjecture}.  The main idea of the proof is to combine the rook placements of Section~\ref{sec:hyperrooks} with the \emph{Bruhat decomposition} of the general linear group.  We begin by establishing some definitions.

\begin{definition}\label{def:inversions}
For a permutation $w \in \Symm_n$, we say that a pair $(i, j)$ of integers with $1 \leq i < j \leq n$ is an \defn{inversion} of $w$ if $w(i) > w(j)$, and we denote by $\Inv(w)$ the \defn{inversion set} 
\[
\Inv(w) = \{(i, j) : (i, j) \text{ is an inversion of } w\}.
\]
\end{definition}

For example, the permutation $w = 25314 \in \Symm_5$ has $\Inv(w) = \{(1, 4), (2, 3), (2, 4), (2, 5), (3, 4)\}$.  Observe that, by definition,
\[
\Inv(w\inv) = \bigl\{ (w(j), w(i)) : (i, j) \in \Inv(w)\bigr\}, 
\]
so in particular the two sets $\Inv(w)$ and $\Inv(w\inv)$ are equinumerous.

\begin{definition}
    For a permutation $w \in \Symm_n$, an $n \times n$ matrix $A$ over a field $\FF$ is a \defn{NW $w$-augmented matrix} if it has the following properties: $A_{w(i),i} = 1$ for all $i \in [n]$, if $(i, j)$ is an inversion of $w$ then the entry $A_{w(j), i}$ can be any element of $\FF$, and all other entries of $A$ are equal to $0$. 
    Similarly, a \defn{SE $w$-augmented matrix} over a field $\FF$ is a matrix $A$ with the following properties: $A_{w(i),i} = 1$ for all $i$, if $(i, j)$ is an inversion of $w$ then the entry $A_{w(i), j}$ can be any element of $\FF$, and all other entries of $A$ are equal to $0$.
    
    We define ${}^*w$ to be the set of all NW $w$-augmented matrices and $w_*$ to be the set of all SE $w$-augmented matrices.
\end{definition}

For example, if $w = 25134$, then
\[
{}^*w = \left \{ \begin{bmatrix}
a&b&1&& \\
1& & & & \\
 &c&&1 & \\
 &d& & &1\\
 &1& & &
\end{bmatrix} \colon a, b, c, d \in \FF\right\}
\] and \[
w_* = \left \{ \begin{bmatrix}
 & &1&& \\
1& &a&& \\
 & &&1& \\
 & & & &1\\
 &1&b&c&d\end{bmatrix}\colon a, b, c, d \in \FF\right\}.
\]

In its most concrete form, the Bruhat decomposition is the following statement.
\begin{thm}[{Bruhat decomposition \cite[\S1.10]{stanley}}]
\label{thm:bruhat}
    Let $U$ represent the group of $n \times n$ invertible upper triangular matrices over $\FF$.  Then $\GL_n(\FF)$ is equal to the disjoint union
    \[
    \GL_n(\FF) = \bigsqcup_{w \in \Symm_n} U w_*,
    \]
    where furthermore each matrix in  $U w_*$ 
    has a unique representation as a product 
    $B A$
    with 
    $B \in U$ and $A \in w_*$.
\end{thm}

\begin{remark}\label{rem:equivalent bruhat}
 There are seven other equivalent versions of Theorem~\ref{thm:bruhat}: writing $L$ for the group of $n \times n$ invertible lower triangular matrices and using the obvious notions and notations $w^*$ and ${}_*w$ for SW and NE augmented matrices, we can replace $U w_*$ in the theorem with any of $w_* L$, ${}^*w U$, $L {}^*w$, $w^* L$, $L w^*$, ${}_*w U$, or $U {}_*w$.  One may show the equivalence of these decompositions by various combinations of the operations of transpose and multiplication on the left or right by the antidiagonal permutation matrix.  These eight different decompositions correspond to eight different orders in which one could do the relevant version of Gaussian elimination: choosing pivots from left to right and eliminating up, or choosing pivots from top to bottom and eliminating left, and so on.
\end{remark}

\begin{remark}\label{remark:bruhat}

In the case that $\FF = \FF_q$ is a finite field, Theorem~\ref{thm:bruhat} provides a direct connection between the rook placements and invertible matrices in Haglund's theorem \cite[Thm.~1]{haglund} (Equation~\eqref{eq:Haglund} in the introduction): it's easy to see that for any $A \in w_*$ and $B \in U$ the matrix $BA$ respects a partition $\lambda$ if and only if $w$ respects $\lambda$.  Letting $S_{\lambda}$ be the set of permutations $w \in \Symm_n$ that respect $\lambda$, it follows the number of invertible matrices in $\GL_n(\FF_q)$ that respect $\lambda$ is
    \[
    \sum_{w \in S_{\lambda}}|U w_*| = \sum_{w \in S_{\lambda}} |U|\cdot |w_*| = (q-1)^nq^{\binom{n}{2}} \sum_{w \in S_{\lambda}}q^{|\Inv(w)|}.
    \]
\end{remark}

Our approach to prove Theorem~\ref{thm:weak form of conjecture} is to mimic the argument of Remark~\ref{remark:bruhat} in the higher-dimensional setting.  Towards this end, we make the following definition.

\begin{definition}\label{def:augHM}
Given permutations $\sigma \in \Symm_{k + 1}$ and $\pi \in \Symm_k$, a {\bf $(\sigma, \pi)$-augmented hypermatrix} is a $2 \times (k + 1) \times k$ hypermatrix of the form 
\[
(A, A') \circ E
\]
where $A \in \sigma_*$ and $A' \in {}^*(\pi\inv)$.
\end{definition}

By the definition of the action of $\GL_{k + 1} \times \GL_k$, the formula in the preceding definition can be written even more concretely as
\[
(A, A') \circ E = \left(A \begin{bmatrix}
    1 & 0 & \cdots & 0 \\
    0 & 1 & \cdots & 0 \\
    \vdots & \vdots & \ddots & \vdots \\
    0 & 0 & \cdots & 1 \\
    0 & 0 & \cdots & 0
\end{bmatrix} (A')\transpose, \quad
A \begin{bmatrix}
    0 & 0 & \cdots & 0 \\
    1 & 0 & \cdots & 0 \\
    0 & 1 & \cdots & 0 \\
    \vdots & \vdots & \ddots & \vdots \\
    0 & 0 & \cdots & 1
\end{bmatrix} (A')\transpose\right),
\]  where the matrix $(A')\transpose$ being multiplied on the right belongs to $^*\pi$.

The next three results extend the first part of Remark~\ref{remark:bruhat} to our setting, showing how to reduce the problem of enumerating nondegenerate hypermatrices respecting a plane partition $P$ to the problem of enumerating augmented hypermatrices respecting $P$.

\begin{prop}\label{prop:pull off U}
Every nondegenerate $2 \times (k + 1) \times k$ hypermatrix over $\FF_q$ can be written in exactly $q - 1$ ways as a composition of the form
    \[
    (B, B') \circ H    
    \]
    where $B \in U \subset \GL_{k+1}(\FF_q)$ and $B' \in L \subset \GL_{k}(\FF_q)$ are triangular and $H$ is a $(\sigma, \pi)$-augmented hypermatrix for some permutations $\sigma \in \Symm_{k + 1}, \pi \in \Symm_k$.
\end{prop}
\begin{proof}
By Theorem~\ref{thm:aitken}, each nondegenerate hypermatrix of format $2 \times (k + 1) \times k$ can be written in exactly $q - 1$ ways as a composition $(M,N) \circ E$ where $(M,N) \in \GL_{k+1}(\FF_q) \times \GL_k(\FF_q)$. By Theorem~\ref{thm:bruhat} and Remark~\ref{rem:equivalent bruhat}, there is exactly one choice of permutations $\sigma \in \Symm_{k + 1}$ and $\pi \in \Symm_k$ and matrices $A, A', B, B'$ such that $A \in \sigma_*$, $A' \in {}^*(\pi\inv)$,  $B \in \GL_{k + 1}(\FF_q)$ is upper-triangular, and $B' \in \GL_k(\FF_q)$ is lower-triangular such that $M = BA$ and $N = B'A'$. Since this is a group action, 
    \[ 
    (M,N) \circ E = (BA, B'A') \circ E = \left( (B, B')\cdot (A, A')\right) \circ E = (B, B') \circ \left( (A, A') \circ E \right),
    \]
with $(A, A') \circ E$ the defining form of a $(\sigma, \pi)$-augmented hypermatrix.  Thus each pair $(M, N)$ corresponds to exactly one composition of the desired kind.  This completes the proof.
\end{proof}

\begin{prop}\label{prop:upper triangular}
Let $P$ be a plane partition inside the box of format $2 \times (k + 1) \times k$, let $H$ be a hypermatrix over a field $\FF$, and let $B \in U \subset \GL_{k + 1}(\FF)$ and $B' \in L \subset \GL_k(\FF)$ be triangular.  Then $H$ respects $P$ if and only if $(B, B') \circ H$ does.
\end{prop}
\begin{proof}
When multiplying any matrix $X$ by an upper-triangular matrix $B$ on the left, the nonzero entries in the product $BX$ all lie at positions that are nonzero in $X$, or directly above such positions (in the same column).  Likewise, when multiplying any matrix $X$ by an upper-triangular matrix $(B')\transpose$ on the right, the nonzero entries in $X(B')\transpose$ all lie at positions that are nonzero in $X$, or directly to the right of such positions (in the same row).  It follows immediately that if $X$ respects the partition $\lambda$ and $B, (B')\transpose$ are upper-triangular, then $BX(B')\transpose$ also respects $\lambda$.  Now the result follows by considering $X$ to be in turn the two faces of $H$.
\end{proof}

\begin{cor}\label{cor:count}
Let $P$ be a plane partition inside the box of format $2 \times (k + 1) \times k$.  The number of nondegenerate hypermatrices over $\FF_q$ that respect $P$ is
\[
q^{k^2}(q - 1)^{2k}
\sum_{(\sigma, \pi) \in \Symm_{k + 1} \times \Symm_k}
\# \{ (\sigma, \pi)\text{-augmented hypermatrices that respect } P \}.
\]
\end{cor}
\begin{proof}
By Proposition~\ref{prop:pull off U}, each nondegenerate hypermatrix can be written in exactly $q - 1$ ways as $(B, B') \circ H$ where $H$ is a $(\sigma, \pi)$-augmented hypermatrix, for some $\sigma \in \Symm_{k + 1}$ and $\pi \in \Symm_k$.  Furthermore, by Proposition~\ref{prop:upper triangular}, the original hypermatrix respects $P$ if and only if $H$ does.  The number of choices for $B, B'$ is $q^{\binom{k + 1}{2}}(q - 1)^{k + 1} \cdot q^{\binom{k}{2}}(q - 1)^k = q^{k^2} (q - 1)^{2k + 1}$.  Thus, there is a $(q - 1)$-to-$q^{k^2} (q - 1)^{2k + 1}$ correspondence between the set of nondegenerate hypermatrices that respect $P$ and the collection (over all $\sigma$ and $\pi$) of $(\sigma, \pi)$-augmented hypermatrices that respect $P$.  The result follows immediately.
\end{proof}

The next result identifies precisely which pairs $(\sigma, \pi)$ contribute to the sum in Corollary~\ref{cor:count}. 

\begin{prop}\label{prop:1s in augmented hypermatrix}
For any plane partition $P$ inside the box of format $2 \times (k + 1) \times k$, there exist $(\sigma, \pi)$-augmented hypermatrices that respect $P$ if and only if the hyperrook placement $(m_\sigma, m_{\pi\inv})\circ E$ respects $P$.
\end{prop}
\begin{proof}
Since $m_\sigma \in \sigma_*$ and $m_{\pi\inv} \in {}^*(\pi\inv)$, if the hyperrook placement $(m_\sigma, m_{\pi\inv})\circ E$ respects $P$ then there is at least one $(\sigma, \pi)$-augmented hypermatrix that respects $P$.  

Conversely, suppose that there is a $(\sigma, \pi)$-augmented hypermatrix $H = (A, (A')\transpose) \circ E$ that respects $P$ (so $A \in \sigma_*$ and $A' \in {}^*\pi$). Recall from  Definition~\ref{wc} the notations $w = \sigma \ol{\pi}$ and $c = \ol{\pi}\inv(1 \ 2 \ \cdots \ k \ k+1) \ol{\pi}$. By definition, we have for each $j \in [k]$ that
\[
H_{1, w(j), j} = \sum_{\ell = 1}^k A_{w(j), \ell} A'_{\ell, j}.
\]
By the definition of $\sigma_*$ and $w$, we have that
$1 = A_{w(j), \sigma\inv w(j)} = A_{w(j), \ol{\pi}(j)} = A_{w(j), \pi(j)}$ is the unique $1$ in row $w(j)$ of $A$ and that $A_{w(j), \ell} = 0$ for $\ell < \pi(j)$.  Similarly, by the definition of $^*\pi$, $A'_{\pi(j), j} = 1$ is the unique $1$ in column $j$ of $A'$ and $A'_{\ell, j} = 0$ for $\ell > \pi(j)$.  Thus $H_{1, w(j), j} = 1$.  Since $H$ respects $P$, the position $(1, w(j), j)$ must not be contained in $P$.  Similarly, we have by definition that for each $j \in [k]$,
\[
H_{2, wc(j), j} = \sum_{\ell = 2}^{k + 1} A_{wc(j), \ell} A'_{\ell - 1, j}.
\]
By the definition of $\sigma_*$, $A_{wc(j), \ol{\pi}c(j)} = A_{wc(j), \sigma\inv wc(j)} = 1$ and $A_{wc(j), \ell} = 0$ if $\ell < \ol{\pi}c(j)$.  Similarly, by the definition of ${}^*\pi$, $A'_{\ol{\pi}(j), j} = 1$ and $A'_{\ell - 1, j} = 0$ if $\ell - 1 > \ol{\pi}(j)$.  Furthermore, since $j \leq k$, we have $\ol{\pi}c (j) = (1 \ 2 \ \cdots \ k + 1) \ol{\pi}(j) = \pi(j) + 1$, and so the previous condition $\ell - 1 > \ol{\pi}(j)$ can be rewritten as $\ell > \ol{\pi}c(j)$.  Thus $H_{2, wc(j), j} = 1$, and since $H$ respects $P$, the position $(2, wc(j), j)$ must not be contained in $P$.  Finally, since $P$ does not contain any of the entries $(1, w(j), j)$ or $(2, wc(j), j)$ for $1 \leq j \leq k$,  it follows from Lemma~\ref{lem:hyperrookWC} that the hyperrook placement $(m_\sigma, m_{\pi\inv})\circ E$ respects $P$.
\end{proof}

To complete the proof of Theorem~\ref{thm:weak form of conjecture}, it suffices to show that the number of $(\sigma, \pi)$-augmented hypermatrices that respect $P$ is a power of $q$ when the hyperrook placement $(m_\sigma, m_{\pi\inv}) \circ E$ respects $P$.  In comparison with the argument in the matrix case (see Remark~\ref{remark:bruhat}), this is significantly complicated by the fact that a generic $(\sigma, \pi)$-augmented hypermatrix may not respect a plane partition $P$, even when the associated hyperrook placement does.

\begin{example}\label{ex:respect}
Consider the case with $k = 4$, $\sigma = 52134$, and $\pi = 2314$.  In this case, the generic element of $\sigma_*$ is
\[
\begin{bmatrix}
    0 & 0 & 1 & 0 & 0 \\
    0 & 1 & x_{2, 3} & 0 & 0 \\
    0 & 0 & 0 & 1 & 0 \\
    0 & 0 & 0 & 0 & 1 \\
    1 & x_{1, 2} & x_{1, 3} & x_{1, 4} & x_{1, 5}     
\end{bmatrix},
\]
the generic element of ${}^*\pi$ is
\[
\begin{bmatrix}
    y_{1,3} & y_{2,3} & 1 & 0 \\
    1 & 0 & 0 & 0 \\
    0 & 1 & 0 & 0 \\
    0 & 0 & 0 & 1      
\end{bmatrix},
\]
and so a generic $(\sigma, \pi)$-augmented hypermatrix is of the form
\[
\left( 
    \begin{bmatrix}
        0 & 1 & 0 & 0 \\
        1 & x_{2, 3} & 0 & 0 \\
        0 & 0 & 0 & 1 \\
        0 & 0 & 0 & 0 \\
        x_{1,2} + y_{1,3} & x_{1,3} + y_{2,3} & 1 & x_{1,4}
    \end{bmatrix}, \quad 
    \begin{bmatrix}
        1 & 0 & 0 & 0 \\
        x_{2,3}+y_{1,3} & y_{2,3} & 1 & 0 \\
        0 & 1 & 0 & 0 \\
        0 & 0 & 0 & 1 \\
        x_{1,3}+x_{1,2}y_{1,3} & x_{1,4}+x_{1,2}y_{2,3} & x_{1,2} & x_{1,5}
    \end{bmatrix}\right).
\]
Although the hyperrook placement $(m_\sigma, m_{\pi\inv}) \circ E$ respects the plane partition $P = \bigl(\langle 3, 2\rangle, \langle 2, 2\rangle\bigr)$, 
the generic $(\sigma, \pi)$-augmented hypermatrix does not.
\end{example}

Example~\ref{ex:respect} shows that the enumeration of augmented hypermatrices respecting a plane partition may be nontrivial.  

\begin{example}\label{eg:solve}
Continue the choices of Example~\ref{ex:respect}.  It follows from the preceding computation that the number of $(\sigma, \pi)$-augmented hypermatrices over $\FF_q$ that respect $P$ is the number of solutions over $\FF_q$ to the system of algebraic equations
\[
\left\{
\begin{aligned}
    x_{1,2} + y_{1,3} & = 0 \\
    x_{1,3} + y_{2,3} & = 0 \\
    x_{1,2}y_{1,3}+x_{1,3} & = 0 \\
    x_{1,2}y_{2,3}+x_{1,4} & = 0
\end{aligned}\right.
\]
in the $|\Inv(\sigma)|+|\Inv(\pi)| = 7$ variables $x_{1,2}, \ldots, y_{2,3}$.  We may solve this system as follows: first, choose the values of $x_{2, 3}$, $x_{1, 5}$, and $x_{1, 2}$ freely ($q^3$ choices).  The first equation in the system determines the value of $y_{1,3}$ uniquely.  Then the third equation determines the value of $x_{1, 3}$ uniquely.  Then the second equation determines the value of $y_{2,3}$ uniquely, and the last equation determines the value of $x_{1, 4}$ uniquely.  So in total there are $q^3$ $(\sigma, \pi)$-augmented hypermatrices over $\FF_q$ in this case.
\end{example}

The final stage in the proof of Theorem~\ref{thm:weak form of conjecture} is to show that the preceding argument can be carried out in general: that is, for each plane partition $P$ and each hyperrook placement $(m_\sigma, m_{\pi\inv})\circ E$ that respects $P$, there exists a special subset of the free entries of $\sigma_*$ and ${}^*\pi$ such that for any choice of values for the elements of this subset, there is a unique $(\sigma, \pi)$-augmented hypermatrix with the special subset taking these values that respects $P$. In order to implement this approach, we use the following labeling scheme for the free entries of an augmented matrix, which is already illustrated above in Example~\ref{ex:respect}.

\begin{definition}\label{def:variables}
Given a permutation $\sigma \in \Symm_{k + 1}$, for each inversion $(i, j)$ of $\sigma$, we write $x_{i, j}$ for the corresponding entry in position $(\sigma(i), j)$ of a generic matrix in $\sigma_*$.  Likewise, given a permutation $\pi \in \Symm_k$, for each inversion $(i, j)$ of $\pi$, we write $y_{i, j}$ for the entry in the corresponding position $(\pi(j), i)$ of a generic matrix in ${}^*\pi$.
\end{definition}

Our next result identifies those positions in the front face of a $(\sigma, \pi)$-augmented hypermatrix that are not identically zero and lie below and to the left of the $1$s in the associated rook placement.  These are precisely the entries that could prevent a $(\sigma, \pi)$-augmented hypermatrix from respecting a plane partition $P$ that is respected by the hyperrook placement $(m_\sigma, m_{\pi\inv})\circ E$.  The statement is illustrated in Figure~\ref{fig:complicated}(a).

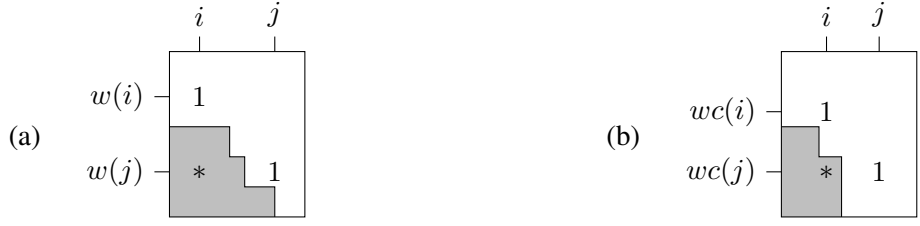
\begin{figure}
    \centering

    \raisebox{2.5em}{(a)}
    \quad 
  \begin{tikzpicture}[scale=.2]
    \fill[lightgray] (0, 0) -- (0, 6) -- (4, 6) -- (4, 4) -- (5, 4) -- (5, 2) -- (7, 2) -- (7, 0) -- cycle;
    \draw (0, 6) -- (4, 6) -- (4, 4) -- (5, 4) -- (5, 2) -- (7, 2) -- (7, 0);
    \draw (0, 0) -- (9, 0) -- (9, 11) -- (0, 11) -- cycle;
    \draw (2, 11) -- (2, 12)
        (7, 11) -- (7, 12)
        (-1, 3) -- (0, 3)
        (-1, 8) -- (0, 8);
    \node[above] at (2, 12) {$i$};
    \node[above] at (7, 12) {$j$};
    \node[left] at (-1, 8) {$w(i)$};
    \node[left] at (-1, 3) {$w(j)$};
    \node at (2, 8) {$1$};
    \node at (7, 3) {$1$};
    \node at (2, 3) {$*$};
    \end{tikzpicture}
    \hspace{1.5in}
    \raisebox{2.5em}{(b)}
    \quad 
    \begin{tikzpicture}[scale=.2]
    \fill[lightgray] (0, 0) -- (0, 6) -- (2.5, 6) -- (2.5, 4) -- (4, 4) -- (4, 0) -- cycle;
    \draw (0, 6) -- (2.5, 6) -- (2.5, 4) -- (4, 4) -- (4, 0);
    \draw (0, 0) -- (9, 0) -- (9, 11) -- (0, 11) -- cycle;
    \draw (3, 11) -- (3, 12)
        (6.5, 11) -- (6.5, 12)
        (-1, 3) -- (0, 3)
        (-1, 7) -- (0, 7);
    \node[above] at (3, 12) {$i$};
    \node[above] at (6.5, 12) {$j$};
    \node[left] at (-1, 7) {$wc(i)$};
    \node[left] at (-1, 3) {$wc(j)$};
    \node at (3, 7) {$1$};
    \node at (6.5, 3) {$1$};
    \node at (3, 3) {$*$};
    \end{tikzpicture}
    \caption{Left: the situation described in Proposition~\ref{prop:potentially bad entries in front face}, with a potential $\lambda$ in gray.  Right: the situation described in Proposition~\ref{prop:potentially bad entries in back face}, with a potential $\mu$ in gray.}
    \label{fig:complicated}
\end{figure}

\begin{prop}\label{prop:potentially bad entries in front face}
Fix permutations $(\sigma, \pi) \in \Symm_{k+1} \times \Symm_k$ and integers $i, j$ such that $1 \leq i < j \leq k + 1$ and $w(i) < w(j)$.  Then $(1,w(j), i)$ is not identically zero in the set of $(\sigma, \pi)$-augmented hypermatrices if and only if $(i,j)$ is an inversion of $\pi$ (equivalently, $j \leq k$ and $\pi(j) < \pi(i)$). Furthermore if $(i,j)$ is an inversion of $\pi$, then the $(1,w(j), i)$-entry is given by the polynomial
\[x_{\pi(j), \pi(i)} + y_{i,j} + \sum_{t \in S_{i,j}}x_{\pi(j), \pi(t)}y_{i,t}\] 
where the $x$- and $y$-variables are as in Definition~\ref{def:variables} and
\[S_{i,j} = \{ t \mid (i, t) \in \Inv(\pi) \text{ and } (\pi(j), \pi(t)) \in \Inv(\sigma)\}.\]
\end{prop}
\begin{proof}
Let $A$ be a generic element of $\sigma_*$ and let $A'$ be a generic element of ${}^*\pi$.
Therefore, by definition, the $(1, w(j), i)$-entry of the $(\sigma, \pi)$-augmented hypermatrix $(A, (A')\transpose) \circ E$ is
\begin{equation}\label{(1, w(j), i)-entry}
\sum_{\ell = 1}^k A_{w(j), \ell} A'_{\ell, i}.
\end{equation}
By the definition of $\sigma_*$ and $w$, we have that
$1 = A_{w(j), \sigma\inv w(j)} = A_{w(j), \ol{\pi}(j)}$ is the unique $1$ in row $w(j)$ of $A$ and that $A_{w(j), \ell} = 0$ for $\ell < \ol{\pi}(j)$.  Similarly, $A'_{\pi(i), i} = 1$ is the unique $1$ in column $i$ of $A'$ and $A'_{\ell, i} = 0$ for $\ell > \pi(i)$.  It follows immediately that if $\pi(i) < \ol{\pi}(j)$, then every term in the sum \eqref{(1, w(j), i)-entry} is equal to $0$, and consequently that the $(1, w(j), i)$-entry of $(A, (A')\transpose)\circ E$ is identically equal to $0$ in this case.

Now suppose $\pi(i) > \ol{\pi}(j)$.  Then $j < k + 1$ and $(i, j)$ is an inversion of $\pi$. Then the formula \eqref{(1, w(j), i)-entry} for the $(1, w(j), i)$-entry of $(A, (A')\transpose) \circ E$ can be rewritten as
\begin{align*}
\sum_{\ell = \pi(j)}^{\pi(i)} A_{w(j), \ell} A'_{\ell, i} & =
A_{w(j), \pi(j)} A'_{\pi(j), i}
+ \Biggl(\sum_{\ell = \pi(j) + 1}^{\pi(i) - 1} A_{w(j), \ell} A'_{\ell, i}\Biggr)
+ A_{w(j), \pi(i)} A'_{\pi(i), i} \\
& = A'_{\pi(j), i} + A_{w(j), \pi(i)} + \sum_{\ell = \pi(j) + 1}^{\pi(i) - 1} A_{w(j), \ell} A'_{\ell, i}.\end{align*} 
Since $(i,j)$ is an inversion of $\pi$, $A'_{\pi(j), i} = y_{i,j}$. Also, since $w(i) < w(j)$, we have \[\sigma(\pi(i)) = w(i) < w(j) = \sigma(\pi(j)),\] so that $(\pi(j), \pi(i))$ is an inversion of $\sigma$. Therefore $A_{w(j), \pi(i)} = x_{\pi(j), \pi(i)}$. Thus we can further simplify our expression for the $(1, w(j), i)$-entry of $(A, (A')\transpose) \circ E$ to 
\[
x_{\pi(j), \pi(i)} + y_{i,j} + \sum_{\ell = \pi(j) + 1}^{\pi(i) - 1} A_{w(j), \ell} A'_{\ell, i} \,,
\]
and in particular we see that this entry is not identically zero.  This completes the proof of the ``if and only if'' part of the statement.  To finish, we need to identify the nonzero terms $A_{w(j), \ell} A'_{\ell, i}$ that appear in the previous sum.  

By the definition of ${}^*\pi$, we know for $\ell < \pi(i)$ that the entry $A'_{\ell, i}$ is generically nonzero if and only if $(i, \pi\inv(\ell))$ is an inversion of $\pi$, and so in particular that $i < \pi\inv(\ell)$.  In this case, the entry is $y_{i, \pi\inv(\ell)}$. Similarly, for $\ell > \pi(j)$, the entry $A_{w(j), \ell} = A_{\sigma(\pi(j)), \ell}$ is generically nonzero if and only if $(\pi(j), \ell)$ is an inversion of $\sigma$, or equivalently if $\sigma(\ell) < w(j)$, and in this case the entry is $x_{\pi(j), \ell}$.  Therefore, defining $t = \pi\inv(\ell)$, the summand $A_{w(j), \ell} A'_{\ell, i}$ in the last displayed equation above is nonzero exactly when
\[
(i, t) \in \Inv(\pi) 
\quad \text{ and } \quad
(\pi(j), \pi(t)) \in \Inv(\sigma),
\]
and in this case its value is $x_{\pi(j), \pi(t)} y_{i, t}$.  This completes the proof.
\end{proof}

We now do a similar analysis for the back face.  The corresponding situation is illustrated in Figure~\ref{fig:complicated}(b).

\begin{prop}\label{prop:potentially bad entries in back face}
Fix permutations $(\sigma, \pi) \in \Symm_{k+1} \times \Symm_k$ and integers $i, j$ such that $1 \leq i < j \leq k + 1$ and $wc(i) < wc(j)$. Then the $(2, wc(j), i)$-entry is not identically zero in the set of $(\sigma, \pi)$-augmented hypermatrices if and only if $(i, j)$ is an inversion of $\ol{\pi}c$.  This is also equivalent to the condition that either $(i, j)$ is an inversion of $\pi$ or $j = k+1$.
Furthermore, if $(i, j)$ is an inversion of $\pi$ then the $(2, wc(j), i)$-entry is given by
\[x_{\ol{\pi} c(j),\ol{\pi} c(i)} + y_{i,j} + \sum_{t \in S_{i,j}'}x_{\ol{\pi} c(j), \ol{\pi} c(t)}y_{i,t},\]
while if $j = k + 1$ then the $(2, wc(j), i)$-entry is given by
\[
x_{\ol{\pi} c(j),\ol{\pi} c(i)} + \sum_{t \in S_{i,j}'}x_{\ol{\pi} c(j), \ol{\pi} c(t)}y_{i,t},
\]
where in both cases \[S'_{i, j} = \left\{t \mid (i,t) \in \Inv(\pi) \text{ and } (\ol{\pi} c(j), \ol{\pi} c(t)) \in \Inv(\sigma)\right\}.\]
\end{prop}

\begin{proof}
Let $A$ be a generic element of $\sigma_*$ and let $A'$ be a generic element of ${}^*\pi$.  Therefore, by definition, the $(2, wc(j), i)$-entry of the $(\sigma, \pi)$-augmented hypermatrix $(A, (A')\transpose) \circ E$ is
\begin{equation}\label{(2, wc(j), i)-entry}
\sum_{\ell = 2}^{k + 1} A_{wc(j), \ell} A'_{\ell - 1, i}.
\end{equation}
By the definition of $\sigma_*$, $A_{wc(j), \ol{\pi}c(j)} = A_{wc(j), \sigma\inv wc(j)} = 1$ and $A_{wc(j), \ell} = 0$ if $\ell < \ol{\pi}c(j)$.  Similarly, by the definition of ${}^*\pi$, $A'_{\pi(i), i} = 1$ and $A'_{\ell - 1, i} = 0$ if $\ell - 1 > \pi(i)$.  Furthermore, since $i < k + 1$, we have $\ol{\pi}c (i) = (1 \ 2 \ \cdots \ k + 1) \ol{\pi}(i) = \pi(i) + 1$, and so the previous condition $\ell - 1 > \pi(i)$ can be rewritten as $\ell > \ol{\pi}c(i)$.  It follows immediately that if $\ol{\pi}c (i) < \ol{\pi}c(j)$, then every term in the sum \eqref{(2, wc(j), i)-entry} is equal to $0$, and consequently that the $(2, wc(j), i)$-entry of $(A, (A')\transpose)\circ E$ is identically equal to $0$ in this case.  

Now suppose $\ol{\pi}c(i) > \ol{\pi}c(j)$, so $(i, j)$ is an inversion of $\ol{\pi}c$.  Since $wc(i) < wc(j)$ (by hypothesis) and $wc = \sigma (\ol{\pi}c)$ (by definition), this means that $(\ol{\pi}c(j), \ol{\pi}c(i))$ is an inversion of $\sigma$, and $A_{wc(j), \ol{\pi}c(i)} = x_{\ol{\pi}c(j), \ol{\pi}c(i)}$.  Since $A'_{\ell - 1, i} = 1$ when $\ell = \pi(i) + 1 = \ol{\pi}c(i)$ and $A'_{\ell - 1, i} = 0$ for $\ell > \pi(i) + 1$,  the formula \eqref{(2, wc(j), i)-entry} for the $(2, wc(j), i)$-entry of $(A, (A')\transpose)\circ E$ can be rewritten as 
\[
x_{\ol{\pi}c(j), \ol{\pi}c(i)} + \sum_{\ell = 2}^{\pi(i)} A_{wc(j), \ell} A'_{\ell - 1, i}.
\]
In particular, we see that this entry is not identically zero in this case.  This establishes the first equivalence, i.e., if $(i,j)$ is an inversion of $\overline{\pi}c$ then the $(2,wc(j), i)$-entry is nonzero.  

For the second equivalence, observe first that $\ol{\pi}c(k + 1) = 1$ and so $(i, k + 1) \in \Inv(\ol{\pi}c)$ for all $1 \leq i \leq k$.  Furthermore, if $j \leq k$ then $\ol{\pi}c(j) = \pi(j) + 1$ and $\ol{\pi}c(i) = \pi(i) + 1$, so $(i, j)$ is an inversion of $\ol{\pi}c$ if and only if $(i, j)$ is an inversion of $\pi$.  This establishes the second equivalence.  

To finish, we need to identify the indices $\ell \in [2, \pi(i)]$ that give rise to nonzero terms $A_{wc(j), \ell} A'_{\ell - 1, i}$ in the last displayed equation above.

If $j < k + 1$ and $(i, j) \in \Inv(\ol{\pi} c)$, then the rest of the proof is similar to the proof of Proposition~\ref{prop:potentially bad entries in front face} after shifting everything by $c$: define $t = \pi\inv(\ell - 1)$ and use the fact that $\ol{\pi}c(t) = 1 + \ol{\pi}(t)$ for $1 \leq t \leq k$.  So suppose instead that $j = k + 1$.  In that case, $\pi c(j) = 1$ and $wc(j) = \sigma(1)$, and the $\ell$th term in the last displayed equation is nonzero exactly when both $\sigma(\ell) < \sigma(1)$ (in which case $A_{wc(j), \ell} = A_{\sigma(1), \ell} = x_{1, \ell} = x_{\ol{\pi}c(j), \ell}$) and $\pi\inv(\ell - 1) > i$ (in which case $A'_{\ell - 1, i} = y_{i, \pi\inv(\ell - 1)} = y_{i, c\inv\pi\inv(\ell)}$).  Defining $t = c\inv\pi\inv(\ell) = \pi\inv(\ell - 1)$, the conditions in the preceding sentence become $wc(t) < wc(j)$ and $t > i$, while the restriction $2 \leq \ell \leq \pi(i)$ becomes $\ol{\pi} c(j) < \pi c(t) < \pi c(i)$.  These conditions are equivalent to the conditions that $(i, t) \in \Inv(\pi) = \Inv(\pi c) \cap [k]^2$ and $(\pi c(t), \pi c(j)) \in \Inv(\sigma)$. This completes the proof. 
\end{proof}

Next, we define a directed graph structure on the variables that appear in a generic $(\sigma, \pi)$-augmented hypermatrix.  In the final proof of Theorem~\ref{thm:weak form of conjecture}, this structure will guide us in choosing a good order to solve for the variables to respect a given plane partition, as in Example~\ref{eg:solve}.

\begin{definition}\label{def:D_sigma, pi}
Fix any $\pi \in \Symm_{k}$ and $\sigma \in \Symm_{k+1}$.  Let 
    \[
    Y_\pi = \{ y_{i,j} : (i,j) \in \Inv(\pi)\}
    \]
be the set of variables in the generic element of $^*\pi$, and let 
 \[
 X_{\sigma} = \{ x_{i,j} : (i,j) \in \Inv(\sigma)\}
 \]
be the set of variables in the generic element of $\sigma_*$.
Define a digraph $D_{\sigma,\pi}$ as follows: the nodes of $D_{\sigma, \pi}$ are the variables $X_{\sigma} \cup Y_\pi$.  Two variables $z_1$, $z_2$ are connected by an arc $z_1 \to z_2$ if either of the following happen: there is some $(i, j) \in \Inv(\pi)$ such that, in the $(1, w(j), i)$-entry of the generic $(\sigma, \pi)$-augmented hypermatrix given by Proposition~\ref{prop:potentially bad entries in front face}, $z_2 = y_{i,j}$ is the element of $Y_\pi$ appearing as a linear term and $z_1$ is any of the other variables that appear; or, there is some $(i, j) \in \Inv(\ol{\pi}c)$ such that, in the $(2, wc(j), i)$-entry of the generic $(\sigma, \pi)$-augmented hypermatrix given by Proposition~\ref{prop:potentially bad entries in back face}, $z_2 = x_{\ol{\pi} c(j), \ol{\pi} c(i)}$ is the element of $X_\sigma$ appearing as a linear term and $z_1$ is any of the other variables that appear.
\end{definition}

\begin{example}
The digraph $D_{52134, 2314}$ corresponding to Example~\ref{ex:respect} contains the following arcs:
\begin{itemize}
 \item $x_{1, 2} \to y_{1,3}$, coming from entry $(1, 5, 1)$, with $(i, j) = (1, 3)$,
 \item $x_{1, 3} \to y_{2,3}$, coming from entry $(1, 5, 2)$, with $(i, j) = (2, 3)$,
 \item $y_{1,3} \to x_{2, 3}$, coming from entry $(2, 2, 1)$, with $(i, j) = (1, 3)$,
 \item $x_{1, 2} \to x_{1, 3}$ and $y_{1,3} \to x_{1, 3}$, coming from entry $(2, 5,1)$, with $(i, j) = (3, 5)$, and 
 \item $x_{1, 2} \to x_{1, 4}$ and $y_{2,3} \to x_{1, 4}$, coming from entry $(2, 5, 2)$, with $(i, j) = (2, 5)$.
\end{itemize} 
\end{example}

By checking the definitions of $S_{i, j}$ and $S'_{i, j}$ in Propositions~\ref{prop:potentially bad entries in front face} and~\ref{prop:potentially bad entries in back face}, we see that every arc in $D_{\sigma, \pi}$ belongs to one of the following six cases: for some $(i, j) \in \Inv(\pi)$ such that $w(i) < w(j)$, it is of the form
\begin{enumerate}[label = (\alph*)]
\item{ $x_{\pi(j), \pi(i)} \rightarrow y_{i,j}$},
\end{enumerate}
or there is a $t \in [k]$ with $\pi(t) < \pi(i)$  for which it is of the form
\begin{enumerate}[label = (\alph*)]\addtocounter{enumi}{1}
   \item $x_{\pi(j), \pi(t)} \rightarrow y_{i,j}$ or
\item $y_{i,t} \rightarrow y_{i,j}$ with $\pi(j) < \pi(t)$,
\end{enumerate}
or for some $(i, j) \in \Inv(\ol{\pi}c)$ such that $wc(i) < wc(j)$ it is of the form
\begin{enumerate}[label = (\alph*)]\addtocounter{enumi}{3}
\item{$y_{i,j} \rightarrow x_{\ol{\pi} c(j), \ol{\pi} c(i)}$},
\end{enumerate}
or there is a $t \in [k]$ with $\pi(t) < \pi(i)$ for which it is of the form
\begin{enumerate}[label = (\alph*)]\addtocounter{enumi}{4}
  \item $y_{i,t} \rightarrow x_{\ol{\pi} c(j), \ol{\pi} c(i)}$ or
   \item $x_{\ol{\pi} c(j), \ol{\pi} c(t)} \rightarrow x_{\ol{\pi} c(j), \ol{\pi} c(i)}$.
\end{enumerate}

Our next result will be used in the final stage of the proof to produce a linear order on the variables that we can use to solve for them one-by-one, as in Example~\ref{eg:solve}.

\begin{prop}\label{prop:acyclic}
For any permutations $\sigma \in \Symm_{k + 1}$, $\pi \in \Symm_k$, the digraph $D_{\sigma, \pi}$ is acyclic.
\end{prop}
\begin{proof}
Consider any (directed) walk $W$ in $D_{\sigma, \pi}$.  We seek to show that $W$ cannot be a cycle.

If $W$ visits only vertices in $X_\sigma$, then the arcs it follows all belong to case (f).  Each such arc is of the form $x_{\ol{\pi} c(j), \ol{\pi} c(t)} \to x_{\ol{\pi} c(j), \ol{\pi} c(i)}$ with $\ol{\pi} c(t)  = \pi(t) + 1 < \pi(i) + 1 = \ol{\pi} c(i)$; since the second coordinate strictly increases at each step, no such walk can form a cycle.

Similarly, if $W$ visits only vertices in $Y_\pi$, then the arcs it follows all belong to case (c).  Each such arc is of the form $y_{i, t} \to y_{i, j}$ with $\pi(j) < \pi(t)$; since the value of $\pi$ applied to the second index strictly decreases at each step, no such walk can form a cycle.

Otherwise, $W$ passes through vertices in both $X_\sigma$ and $Y_\pi$.  Without loss of generality, we may assume that $W$ starts and ends in $Y_\pi$, so for some $m \geq 1$ we can write $W$ as
 \[W_1 \to W'_1 \to W_2 \to W'_2 \to \cdots \to W_m \to W'_m \to W_{m + 1}\]
where $W_1, W_2, \ldots, W_{m + 1}$ are walks (possibly of length $0$) on the subgraph of $D_{\sigma, \pi}$ induced by $Y_\pi$ and $W'_1,W'_2,\ldots, W'_m$ are walks on the subgraph induced by $X_{\sigma}$.  We define a weight function $\wt: X_{\sigma}\cup Y_\pi \to \ZZ_{>0}$ by 
 \[
   \wt(y_{i,j}) = \pi(i) 
   \qquad \text{ and } \qquad
   \wt(x_{i,j}) = j.
 \]
By checking each of the six cases into which arcs of $D_{\sigma, \pi}$ can fall, we see that if $z_1 \to z_2$ in $D_{\sigma, \pi}$ then $\wt(z_1) \leq \wt(z_2)$, and moreover that if $z_1 \in Y_\pi$ and $z_2 \in X_\sigma$ then actually $\wt(z_1) < \wt(z_2)$ (specifically parts (d) and (e) of the definition, since $\ol{\pi} c(i) = \pi(i) + 1$ for $i \in [k]$).  Since $W$ contains at least one edge from a $y$-variable to an $x$-variable, it also cannot be a cycle.  This completes the proof.
\end{proof}

\begin{thm}\label{thm:power of q}
For any permutations $(\sigma, \pi) \in \Symm_{k + 1} \times \Symm_k$ and any plane partition $P$ of format $2 \times (k + 1) \times k$, the number of $(\sigma, \pi)$-augmented hypermatrices over $\FF_q$ that respect $P$ is either $0$ or a power of $q$, where in the latter case the exponent depends on $\sigma$, $\pi$, and $P$ but is independent of $q$.
\end{thm}
\begin{proof}
Fix $P$ and $(\sigma, \pi)$.  If the hyperrook placement $(m_\sigma, m_{\pi\inv})\circ E$ does not respect $P$, then by Proposition~\ref{prop:1s in augmented hypermatrix} there are no $(\sigma, \pi)$-augmented hypermatrices that respect $P$.  So suppose instead that $(m_\sigma, m_{\pi\inv})\circ E$ respects $P$.

By Proposition~\ref{prop:acyclic}, the digraph $D_{\sigma, \pi}$ is acyclic; therefore it induces a partial order on the set of variables.  Choose any linear extension of this partial order.  This is a total ordering of the variables in $X_\sigma \cup Y_\pi$ such that, if $z_1$ and $z_2$ are two of these variables such that $z_1 \to z_2$ (as in Definition~\ref{def:D_sigma, pi}), then $z_1$ will precede $z_2$ in this total order.  

Let $A$ be a generic $(\sigma, \pi)$-augmented hypermatrix, and consider an entry $A_{a_1,a_2,a_3}$ of $A$ that is not identically zero and that lies inside $P$. The total order in the previous paragraph induces a total order on the variables that appear in $A_{a_1,a_2,a_3}$. By the construction of $D_{\sigma, \pi}$, the \emph{last} variable in $A_{a_1,a_2,a_3}$ with respect to the total order is a linear term.  Furthermore, it is easy to see from Definition~\ref{def:D_sigma, pi} and Propositions~\ref{prop:potentially bad entries in front face} and~\ref{prop:potentially bad entries in back face} that the last variables for different entries of $A$ are distinct.  Therefore, if we count the assignments of the variables in $X_\sigma \cup Y_\pi$ to elements of $\FF_q$ with the property that the resulting $(\sigma, \pi)$-augmented hypermatrix respects $P$, assigning values to variables in the order given by the total order, we have $q$ choices for each variable that is not the last in its entry (i.e., no restriction on the choice) and $1$ choice for each variable that is the last in its entry, if that entry lies in $P$ (namely, the (unique) value that makes the (unique) entry in which it is the last variable equal to $0$, given the previous choices).  Since the number of choices for each assignment is independent of the choices made, the total number of valid assignments is the product of the number of choices for each variable.  This is always a power of $q$ (with exponent depending only on $\sigma, \pi, P$), as claimed.
\end{proof}

We are now ready to prove Theorem~\ref{thm:weak form of conjecture}, which we restate for convenience.

\weaktheorem*

\begin{proof}
Define for each prime power $q$ and each plane partition $P$ of format $2 \times (k + 1) \times k$ the function
\begin{equation}
\label{eq:f}    
\widetilde{f}(q) = \sum_{(\sigma, \pi) \in \Symm_{k + 1} \times \Symm_k}
\# \{ (\sigma, \pi)\text{-augmented hypermatrices over $\FF_q$ that respect } P \}.
\end{equation}
By Corollary~\ref{cor:count}, the number of nondegenerate hypermatrices over $\FF_q$ that respect the plane partition $P$ is
$q^{k^2} (q - 1)^{2k}\cdot \widetilde{f}(q)$.  By Proposition~\ref{prop:1s in augmented hypermatrix}, the $(\sigma, \pi)$-summand in \eqref{eq:f} is nonzero exactly when $(m_\sigma, m_{\pi\inv})\circ E$ respects $P$; and in that case, by Theorem~\ref{thm:power of q}, the summand is a power of $q$ (whose exponent depends on $\sigma$, $\pi$, and $P$, but not on $q$).  It follows that there is a polynomial $f(x) \in \ZZ_{\geq 0}[x]$ such that $\widetilde{f}(q) = f(q)$ for all prime powers $q$, and furthermore that $f(1)$ is equal to the number of pairs $(\sigma, \pi) \in \Symm_{k + 1}\times \Symm_k$ such that the hyperrook placement $(m_\sigma, m_{\pi\inv})\circ E$ respects $P$.  Then the result follows directly from Theorem~\ref{thm:hyperrookCount}.
\end{proof}

\section{Further Research}
\label{sec:final remarks}

We end with some open questions and some ideas for future research. 

\subsection{Proving the main conjecture}
The main problem raised by our work is to prove Conjecture~\ref{mainConjecture} in full generality.  One approach would be to build on our proof of Theorem~\ref{thm:weak form of conjecture}, which establishes that the number of nondegenerate hypermatrices over $\FF_q$ that respect the plane partition $P$ is
\[
q^{k^2}(q - 1)^{2k} \sum_{\substack{(\sigma, \pi) \in \Symm_{k + 1} \times \Symm_k: \\ (m_\sigma, m_{\pi\inv})\circ E \text{ respects } P}} q^{|\Inv(\sigma)| + |\Inv(\pi)| - h(\sigma, \pi; P)}
\]
where $h(\sigma, \pi; P)$ is defined to be the number of entries in a generic $(\sigma, \pi)$-augmented hypermatrix that are not identically zero and lie inside $P$; these are characterized by Propositions~\ref{prop:potentially bad entries in front face} and~\ref{prop:potentially bad entries in back face}.  Thus, the main conjecture could be proved with ``only combinatorics'', if one could understand the quantity $h(\sigma, \pi; P)$ well enough to find a way to group terms correctly in order to give the desired factored form.  Our attempts to carry this out have thus far been unsuccessful.

\subsection{Other formats}

To what extent do our results extend to other formats?  Our work relies in an essential way on Lemma~\ref{lem:faceSum} and Theorem~\ref{thm:aitken} of Aitken, which hold respectively only for three-dimensional boundary formats and the specific format $2 \times (k + 1) \times k$.  Is there an analogous version of Lemma~\ref{lem:faceSum} for higher-dimensional hypermatrices? If so, that might allow one to extend Theorem~\ref{thm:MaxPlanePartition} (which identifies the unique maximal shape that allows nondegenerate hypermatrices for all three-dimensional boundary formats) to higher dimensions.  A simple dimension count shows that we should \emph{not} expect Theorem~\ref{thm:aitken} to extend to any larger formats;\footnote{For example, the space of $3 \times 3 \times 5$ hypermatrices is $45$-dimensional, whereas the group $\GL_3\times\GL_3\times\GL_5$ is only $9 + 9 + 25 = 43$-dimensional; the space of nondegenerate hypermatrices is presumably full-dimensional, so we should expect a two-parameter family of orbits of nondegenerate hypermatrices of this format under the larger group.} could it nonetheless still be the case that the nondegenerate hypermatrices respecting a plane partition enumerate nicely?  It seems challenging to test this question computationally.

\subsection{Intrinsic definition of hyperrook placements}\label{sec:intrinsic}

Full-rank rook placements on a board can be described as arrangements having one rook in each row and column; they can also be described as the result of multiplying the identity matrix by a permutation matrix. Our Definition~\ref{def:hyperrook placement} of hyperrook placements is analogous to the second definition.  Do the same objects have a natural description along the lines of the first definition?  

It follows immediately from Definition~\ref{def:hyperrook placement} that every hyperrook placement has a single $1$ in each column of each face and at most $1$ in each row of each face.  One is tempted to guess that the hyperrook placements are simply pairs $(S_1,S_2)$ where $S_1,S_2$ are placements of $k$ rooks on the rectangle $\mathcal{R}_{(k+1) \times k}$ with empty intersection and not sharing the same empty row.  However, it's easy to see that the number of such pairs is strictly larger than the number of hyperrook placements when $k > 2$.

There are various definitions of three-dimensional rook placement in the literature that do \emph{not} agree with ours (e.g., Latin squares, or the notion in \cite{3d-2, 3d-1, 3d-3}); typically, these are on cubical boards.

\subsection{Lower rank}

In the case of square matrices, the nondegeneracy condition (that the determinant is nonzero) can be equivalently characterized as a rank condition (that the matrix has full rank), and the notion of a rook placements, $q$-rook numbers, etc., extends to lower ranks in a natural way, placing fewer than the maximum number of rooks and thinking of rook placements as partial permutations.  One might ask whether there is similarly a lower-rank version of our theory.  There are numerous inequivalent definitions of the rank of a hypermatrix (see, e.g., \cite{tensorranks} for a clear description of four such notions).  Although these have sometimes been discussed in the context of the hyperdeterminant (e.g., \cite{Ottaviani, AY}), as far as we know, none of them has a tight connection to our notion of nondegeneracy. Moreover, basic questions about tensor rank (computing the rank of a hypermatrix, determining the orbits under the $\GL(\FF) \times \cdots \times \GL(\FF)$ action, or counting hypermatrices by rank over a finite field) are all difficult, see, e.g., \cite{rank-3, rank-4, rank-1, rank-2, rank-5}.  It is unclear whether a nice lower-rank theory can be based on any of these ideas.

Alternatively, one might hope to construct a lower-rank hyperrook theory along the following lines (mentioned briefly in the introduction): for a partition-shaped board $B$ in the $n \times n$ square, its rook numbers $r_i(B)$ (the number of ways of placing $i$ rooks on $B$) and hit numbers $h_i(B)$ (the number of ways of placing $n$ rooks on the square such that $i$ of them lie in $B$) are related by 
\begin{equation}
\label{eq:rook-hit duality}    
\sum_{i} r_i(B) \cdot (x - 1)^i \cdot (n - i)!  = \sum_{i} h_i(B) \cdot x^i
\end{equation}
(and a $q$-analogue of this identity holds for $q$-rook and $q$-hit numbers).  The definition of hit number extends immediately to our setting, by considering hyperrook placements on the $2 \times (k + 1) \times k$ box.  If we write down the generating function for these numbers, is there a natural change of basis (analogous to \eqref{eq:rook-hit duality}) so that the coefficients can be interpreted as lower-rank rook numbers?

\section*{Acknowledgements}
This paper is based in part on the first author's undergraduate honors thesis \cite{seniorthesis}.  The first author thanks the George Washington University Enosinian Scholars Program and its participants for helpful discussions.  We thank Joe Bonin and Sam F.\ Hopkins for their comments on the thesis version.  Work of the second author was supported in part by a gift from the Simons Foundation (MPS-TSM-00006960).

\bibliographystyle{alpha}
\bibliography{HMBib}

\end{document}